\documentclass[11pt]{article}

\usepackage{amsmath,amsthm}

\usepackage{amssymb,latexsym}

\usepackage{enumerate}

\newtheorem{tw}{Theorem}[section]
\newtheorem{lm}[tw]{Lemma}
\newtheorem{wn}[tw]{Corollary}
\newtheorem{stw}[tw]{Proposition}
\newenvironment{dow}{\it Proof.\rm}{\hfill $\Box$}

\theoremstyle{definition}
\newtheorem{df}[tw]{Definition}
\newtheorem{uw}[tw]{Remark}
\newtheorem{prz}[tw]{Example}

\newcommand{\BR}{{\mathbb R}}
\newcommand{\BX}{{\mathbb X}}

\newcommand{\FF}{{\mathcal{F}}}
\newcommand{\GG}{{\mathcal{G}}}

\newcommand{\BB}{{\mathcal{B}}}

\newcommand{\MM}{{\mathcal{M}}}

\newcommand{\RR}{{\mathcal{R}}}
\newcommand{\EE}{{\mathcal{E}}}

\newcommand{\TT}{{\mathcal{T}}}

\numberwithin{equation}{section}

\begin{document}
\title{Trace operator and the Dirichlet problem for elliptic equations on arbitrary bounded open sets\footnote{This work was supported by Polish Science Centre (grant no. 2016/23/B/ST1/01543).}}
\author{Tomasz Klimsiak}
\date{}
\maketitle

\begin{abstract}
We consider the Dirichlet problem on general, possibly nonsmooth bounded domain, for elliptic linear equation with  uniformly elliptic divergence form operator. We investigate carefully the relationship  between weak, soft and the Perron-Wiener-Brelot solutions of the problem. To this end, we extend the usual notion of the trace operator to Sobolev space  $H^1(D)$ with $D$ being an arbitrary  bounded open subset of $\BR^d$. In the second part of the paper, we prove some existence results for the Dirichlet problem  for semilinear equations with measure data on the right-hand side and $L^1$-data on the Martin boundary of $D$.
\end{abstract}



\footnotetext{T. Klimsiak: Institute of Mathematics, Polish
Academy of Sciences, \'Sniadeckich 8, 00-956 Warszawa, Poland, and
Faculty of Mathematics and Computer Science, Nicolaus Copernicus
University, Chopina 12/18, 87-100 Toru\'n, Poland. {\em E-mail address:} tomas@mat.umk.pl}


\section{Introduction}
\label{sec1}

Let $D$ be a  bounded open subset of $\BR^d$, $d\ge 2$. The main purpose of the first part of the present paper is to investigate the relationship between solutions of the weak Dirichlet problem: for given $\psi\in
H^1(D)$ find $u\in H^1(D)$ such that
\begin{equation}
\label{eq1.2} -Au=0\quad \mbox{in}\quad D,\quad u-\psi\in H^1_0(D)
\end{equation}
(problem wDP$(A,D,\psi)$ for short) and solutions of the
Dirichlet problem, which formally can be formulated as follows:
for given measurable  $\psi:\partial
D\rightarrow\BR$   find
$u\in H^1_{loc}(D)$ such that
\begin{equation}
\label{eq1.3} -Au=0\quad \mbox{in}\quad D,\quad
u=\psi\quad\mbox{on}\quad\partial D
\end{equation}
(problem DP$(A,\partial D,\psi)$ for short). We stress that,
contrary to (\ref{eq1.2}), in (\ref{eq1.3}) the boundary data
$\psi$  are given only on $\partial D$. In the paper we consider weak, soft and Perron-Wiener-Brelot (PWB-solutions for short) solutions to (\ref{eq1.3}). In the case where
$D$ is irregular, careful analysis of  the relationship between these notions of solutions of (\ref{eq1.3}) and solutions of (\ref{eq1.2}) requires the study of more general then (\ref{eq1.3}) Dirichlet problem
\begin{equation}
\label{eq1.19}
-Au=0\quad \mbox{in}\quad D,\quad
u=\psi\quad\mbox{on}\quad\partial_M D
\end{equation}
where $\psi:\partial_MD\rightarrow\BR$ and $\partial_M D$ is the Martin boundary of $D$. Therefore, in fact, in the paper we also consider problem (\ref{eq1.19}).

In the second part of the paper we apply the results of the first part to study the Dirichlet problem for semilinear equations with general measure data on the right-hand side and $L^1$-boundary data.

In (\ref{eq1.2}) and (\ref{eq1.3}), $A$ is a divergence form
operator
\begin{equation}
\label{eq1.1} Au=\sum_{i,j=1}^d (a_{ij}u_{x_i})_{x_j}
\end{equation}
and the equation $-Au=0$ holds in the weak sense, i.e.
\begin{equation} \label{eq1.2.0}
\EE(u,v):=(a\nabla u,\nabla v)_{L^2(D;m)}=0,\quad v\in C^1_c(D).
\end{equation}
In the whole paper we assume that
$a=(a_{ij})_{i,j=1,\dots,d}:\BR^d\rightarrow\BR^d\otimes\BR^d$ is
a bounded symmetric matrix-valued measurable function such that
for some $\lambda>0$,
\begin{equation}
\label{eq1.12}
\sum^d_{i,j=1}a_{ij}(x)\xi_i\xi_j\ge\lambda|\xi|^2,\quad
x\in\BR^d,\,\xi=(\xi_2,\dots,\xi_d)\in\BR^d.
\end{equation}

Our problems concerning the linear  Dirichlet problem are
classical. The relationship between solutions of (\ref{eq1.2}) and
(\ref{eq1.3}) is quite well understood in the case where $D$ has regular boundary $\partial D$ and $\psi\in C(\partial D)$. In the present paper we concentrate on the case where $D$ is an arbitrary open set. The second goal is to extend the existing theory to possibly discontinuous boundary
data. To achieve our goals, we extend the usual notion of the trace operator to the space $H^1(D)$ with general bounded open subset $D$ of $\BR^d$. In fact, in order to study semilinear equations with measure data, we extend the trace operator to even wider space $\TT$ defined  later on.

To describe the content of the paper, we must first explain what we
mean by a solution to (\ref{eq1.3}) in case $\partial D$ is
irregular. We start with recalling some classical results for
$\psi\in C(\partial D)$. Let $m$ denote the Lebesgue measure on
$\BR^d$. It is well know that for each $\psi\in H^1(D)$ there exists a unique solution $u\in H^1(D)$ of the weak Dirichlet problem  (\ref{eq1.2}) and that $u$ has an $m$-version belonging to $C(D)$ (see
\cite{Nash}). Moreover, if $\psi_1-\psi_2\in H^1_0(D)$, then
the solution of wDP$(A,D,\psi_1)$ is equal to the solution of  wDP$(A,D,\psi_2)$. Therefore, we can define the  positive linear  operator
\[
B:H^1(D)/H^1_0(D)\rightarrow H^1(D)
\]
which assigns to each $\psi\in H^1(D)/H^1_0(D)$ the unique solution of
(\ref{eq1.2}). By \cite{LSW}, $B$ is continuous, and moreover, for
$\psi\in (H^1(D)/H^1_0(D))\cap C(\overline{D})$, we have
\begin{equation}
\label{eq1.2.1} \|B\psi\|\le C(\lambda,D)\max_{\partial D}|\psi|,
\end{equation}
where for $u\in H^1_{loc}(D)$, we write
\[
\|u\|= \sup_{D'\subset\subset D} \mbox{dist}(D',D)\|\nabla
u\|_{L^2(D',m)}+\sup_{D}|u|.
\]
Let $\mathcal{X}=\{u\in H^1_{loc}(D)\cap C(D):\|u\|<\infty\}$.  By
(\ref{eq1.2.1}), we can extend  $B$ to a positive linear continuous
operator
\begin{equation}
\label{eq1.15} B: C(\partial D)\rightarrow \mathcal{X}.
\end{equation}
For given $\psi\in C(\partial D)$, the function
$B\psi\in\mathcal{X}$ is called a weak solution of (\ref{eq1.3}). Of course, the solution $u:=B\psi$ to (\ref{eq1.3}) satisfies the second condition in (\ref{eq1.3}) only formally because in general, $u$ is not continuous up to the boundary (unless $D$ is regular).

To  encompass broader class of boundary data, in the
paper we propose a definition of a solution to (\ref{eq1.3}) based
on the notion of the  harmonic measure. Solutions in the sense of
this definition will be called soft solutions. To fix notation,
recall that for a given $x\in D$ the harmonic measure associated
with $x$, $D$ and the operator $A$ is the unique Borel measure on $\partial D$ such that
\begin{equation}
\label{eq.bb}
B\psi(x)=\int_{\partial D}\psi(y)\,\omega^A_{x,D}(dy), \quad
\psi\in C(\partial D).
\end{equation}
(the measure $\omega^A_{x,D}(dy)$ exists and is unique because $B$
defined by (\ref{eq1.15}) is positive and continuous). If there is
no ambiguity, we drop $D$ in the notation. Note that, by  Harnack's
inequality (see \cite{Moser}), for any $x,y$ in the same connected
component of $D$, the  measures $\omega^A_{x}$ and $\omega^A_{y}$ are
mutually absolutely continuous and that it may happen that
$\omega^A_{x}$ is completely singular with respect to the surface
Lebesgue measure $\sigma$ on $\partial D$ even if $D$ is smooth
(see \cite{CFK,ModicaMortola}).

Let
$\{G^D_\alpha,\,\alpha>0\}$ denote the resolvent  operator of $A$ on $D$ (with zero Dirichlet condition), and for a
positive $f\in L^2(D;m)$, let $G^Df=\sup_{\alpha>0}G^D_\alpha f$.
Set $\delta= G^D1$  and
\begin{equation}
\label{eq1.16}
\breve H^1_\delta(D)=\{u\in H^1_{loc}\cap L^2(D;m):\|u\|_{\breve H^1_\delta}<\infty\},
\end{equation}
where $\|u\|_{\breve H^1_\delta}=
\|u\|_{L^2(D;m)}+\|\sqrt\delta \nabla u\|_{L^2(D;m)}$.
By a soft solution to (\ref{eq1.3}) with
$\psi\in L^2(\partial D; \omega^A_m)$, where
\begin{equation}
\label{eq1.13}
\omega^A_m(dy)=\int_D\omega^A_x(dy)\,m(dx),
\end{equation}
we mean a unique function
$u\in \breve H^1_\delta(D)\cap C(D)$ such that
\[
\EE(u,v)=0,\quad v\in C^1_c(D),
\]
and for some (and hence every) increasing sequence of open sets
$\{D_n\}\subset\subset D$ with $\bigcup_{n\ge 1} D_n=D$, we have
\begin{equation}
\label{eq1.17}
\int_{\partial D_n} u(y)\omega^A_{x,D_n}(dy)\rightarrow \int_{\partial D}\psi(y)\omega^A_{x,D}(dy)
\end{equation}
for $m$-a.e. $x\in D$. In the paper we show that,  for $\psi\in
C(\partial D)$, soft solutions of (\ref{eq1.3}) are weak solution
of (\ref{eq1.3}).

Our definition of soft solution to (\ref{eq1.3}) resembles the
definition adopted in the literature in the case of regular
domains and regular coefficients $a_{ij}$  (see
\cite{Chabrowski,MarcusVeron}). The difference is that in these
papers the harmonic measures in (\ref{eq1.17}) are replaced by the
surface measure $\sigma$.  In the case where the domain and the
coefficients are regular, such a modification of the definition of
a solution is possible because then the  harmonic measures are
absolutely continuous with respect to the surface measure (see
\cite{Dahlberg,FJK,Fefferman}).

Another way of defining a solution to (\ref{eq1.3}) is  the method of sub and
superharmonic functions
(so called Perron-Wiener-Brelot method, see \cite{BH}). Recall that, for a given measurable $\psi:\partial D\rightarrow\BR$, a function $u$ is called  a PWB-solution of (\ref{eq1.3}) if
\begin{align*}
u&=\sup\{v: v \mbox{ is subharmonic and }
\limsup_{y\rightarrow x}v(y)\le \psi(x),\, x\in\partial D  \}\\
& =\inf\{v:v \mbox{ is superharmonic and }
\liminf_{y\rightarrow x}v(y)\ge \psi(x),\, x\in\partial D\}.
\end{align*}
We prove that, if $\psi\in C(\partial D)$, then PWB-solution of (\ref{eq1.3}) satisfies
\begin{equation}
\label{eq1.4}
u(x)=\int_{\partial D}\psi(y)\,\omega^A_{x}(dy),\quad x\in D.
\end{equation}
Therefore,  if  $\psi\in C(\partial D)$ then  weak and PWB-solutions
of (\ref{eq1.3}) coincide. Our result generalizes the corresponding result from \cite{ArendtDaners} proved in the case where $A=\Delta$ (see also
\cite{Hildebrandt,Simader}). In fact, we show that, if
$\psi\in\BB(\partial D)$ and the right-hand side of (\ref{eq1.4})
is finite for every $x\in D$, then the PWB-solution of
(\ref{eq1.3}) exists and is given by (\ref{eq1.4}). It is worth
mentioning that our  result
(and the corresponding one in \cite{ArendtDaners}) follows easily from a result proved in \cite{BH} and the fact that
\[
\omega^A_{x,D}(dy)=P_x(X_{\tau_D}\in dy),\quad x\in D,
\]
where $\mathbb X=(\{X_t,\, t\ge 0\},\, \{P_x,\,x\in \BR^d\}
,\,\{\FF_t,\, t\ge 0\})$ is a
diffusion process associated with the operator $A$, and
\begin{equation}
\label{eq1.14}
\tau_D=\inf\{t>0:X_t\in \BR^d\setminus D\}.
\end{equation}
However, our proof is much more elementary. In contrast to
\cite{BH}, it is not based on the abstract theory of balayage
spaces.

We now briefly describe our main results on the relation between
solutions of (\ref{eq1.2}) and (\ref{eq1.3}). We first assume that
$D$ is a Lipschitz domain. Let ${\rm Tr}:H^1(D)\rightarrow
L^2(D;\sigma)$ denote the trace operator. In the paper we show
that, if $\psi\in H^1(D)$, then the solution $u$ of (\ref{eq1.2}) has the representation
\begin{equation}
\label{eq1.5}
u(x)=\int_{\partial D}\widetilde{{\rm Tr}(\psi)}(y)\,\omega^A_{x}(dy),
\end{equation}
where $\widetilde{\rm{Tr}(\psi)}$ is a $\sigma$-version of the trace $\rm{Tr}(\psi)$ of $\psi$ (determined  $\omega^A_{m}$-a.e.) defined as
\[
\widetilde{{\rm Tr}(\psi)}:=\bar\psi_{|\partial D},
\]
where $\bar\psi$ is a quasi-continuous version of an extension of $\psi$ to $H^1(\BR^d)$. In different words, if
$u$ is a solution to wDP($A,D,\psi$), then $u$ is a PWB-solution to DP$(A,\partial D, \widetilde{{\rm Tr}(\psi)})$.
It is worth mentioning that a version $\bar\psi$ does not depend on the operator $A$ since by (\ref{eq1.12}),
\[
\lambda\mbox{Cap}_\Delta\le\mbox{Cap}_A\le \|a\|_{\infty}\mbox{Cap}_\Delta.
\]
Observe that in (\ref{eq1.5}), one can not take  any
$\sigma$-version of $\rm{Tr}(\psi)$ because as mentioned before, in
general, $\omega^A_{x,D}$ is not absolutely continuous with
respect to the surface measure $\sigma$. This is one of the main
differences between PWB-solutions and weak solutions. Namely,
contrary to PWB-solutions, weak solutions to the Dirichlet problem
do not depend on the $\sigma$-version of the boundary data.

In the general case, when $D$ is an arbitrary bounded open set, to
establish a  relation between weak Dirichlet problem and Dirichlet
problem, we introduce the trace operator on $H^1(D)$. It is well
known that there exists a positive function
$g_D\in\BB(D\times D)$, called Green function for $D$ (and $A$)  such that
\[
G^Df=\int_Dg_D(\cdot,y)f(y)\,m(dy)\quad m\mbox{-a.e.}
\]
for any positive $f\in L^2(D;m)$ and the functions $g_D(x,\cdot)$ and $g_D(\cdot,y)$ are  excessive
for all $x,y\in D$. Set $\kappa(x,y)=g_D(x,y)\delta^{-1}(y)$ and
define the metric $\varrho$ on $D$ by
\begin{equation}
\label{eq1.6}
\varrho(x,y)=\sum_{n\ge 1}2^{-n}\frac{|\hat{\kappa}f_n(x)-\hat\kappa f_n(y)|}{1+|\hat{\kappa}f_n(x)-\hat\kappa f_n(y)|},
\end{equation}
where $\hat{\kappa}f_n=\int\kappa (x,\cdot)f_n(x)\,m(dx)$  and
$\{f_n\}$ is a dense subset of $C_0(D)$. Let $D^*$ be the
completion of $D$ with respect to the metric $\varrho$, and let
$\partial_M D:=D^*-D$ ($\partial_M D$ is the so called Martin
boundary of $D$). We define the harmonic measure on the Martin
boundary by
\[
h^A_{x}(dy)=P_x(X_{\tau_D-}\in dy),\quad x\in D,
\]
where $X_{\tau_D-}=\lim_{t\nearrow \tau_D}X_t$ and the  limit is
taken with respect to the metric $\varrho$. Similarly to
(\ref{eq1.13}), we put $h^A_{m}(dy)=\int h^A_{x}(dy)\,m(dx)$. We
prove that there exists a trace operator
\[
\gamma_A:H^1(D)\rightarrow L^2(\partial_M D; h^A_m),
\]
i.e. a  continuous linear operator such that $\gamma_A(\psi)=\psi_{|\partial D}$ for $\psi\in C(\overline D)\cap H^1(D)$.
The last equality  is meaningful since we show that there exists an imbedding
\begin{equation}
\label{eq1.7}
i_M: L^2(\partial D;\omega^A_{m})\hookrightarrow L^2(\partial_M D; h^A_m).
\end{equation}
More precisely, for every  $\psi\in L^2(\partial D;\omega^A_{m})$ there exists a unique
function $i_M(\psi)\in L^2(\partial_M D; h^A_m)$ such that
\[
\int_{\partial D} \psi(y)\, \omega^A_x(dy)=\int_{\partial_M D}i_M(\psi)(y)\,h^A_x(dy),\quad x\in D.
\]
We also prove that the trace theorem holds for the operator $\gamma_A$, i.e.
\[
\gamma_A(\psi)=0\quad\mbox{if and only if}\quad\psi\in H^1_0(D).
\]
It is interesting that in spite of the fact that the trace operator $\gamma_A$ depends on  $A$, the above result  holds independently of $A$.

In general, the embedding  (\ref{eq1.7}) is strict. We show that
\begin{equation}
\label{eq1.8}
\gamma_A^{-1}(L^2(\partial D; \omega^A_m))=H^1_c(D),
\end{equation}
where $H^1_c(D)$ is the set of those $u\in H^1(D)$ for which there exists $\psi\in\BB(\partial D)$ such that
\[
[0,\tau_D]\ni t\rightarrow (u\mathbf{1}_D+\psi\mathbf{1}_{\partial D})(X_t)\mbox{ is continuous }P_x\mbox{-a.s. for q.e. } x\in D,
\]
where q.e. is the abbreviation for quasi everywhere with respect to the capacity Cap$_A$.
The space $H^1_c(D)$ is a closed subspace of $H^1(D)$
and
\[
\mbox{cl}(H^1(D)\cap C(\overline{D}))\subset H^1_c(D).
\]
An immediate  corollary to (\ref{eq1.8}) is that  $L^2(\partial D;\omega^A_{m})$ is isomorphic to $L^2(\partial_M D; h^A_m)$ if and only if $H^1_c(D)=H^1(D)$.
Equivalently, there exists a trace operator on $H^1(D)$ to $L^2(\partial D;\omega^A_{m})$) if and only if $H^1_c(D)=H^1(D)$.
By using this trace operator we show that, if $u$ is a solution of wDP$(A,D,\psi)$, then  it is a weak solution of the  Dirichlet problem
\[
-Au=0\quad\mbox{in}\quad D,\quad u=\gamma_A(\psi)\quad\mbox{on}\quad \partial_M D
\]
(problem DP$(A,\partial_M D,\gamma_A(\psi))$ for short). More precisely,
$u\in H^1(D)$ and
\[
\EE(u,v)=0,\quad v\in H^1_0(D),\quad \gamma_A(u)=\gamma_A(\psi).
\]

Boundary behaviour of a quasi-continuous version of a function
from $H^1(D)$ was for the first time considered by Doob
\cite{Doob} in the case where $A=\Delta$  (see also
\cite{Oksendal,Stoica}). A substantial part of our paper is
devoted to extension of Doob's results to operators of the form
(\ref{eq1.1}) and wider class of functions and to investigate
properties of the trace operator $\gamma_A$. A similar in spirit
(but purely analytic)approach to the definition of the trace
operator on arbitrary bounded domains is considered in
\cite{Shvartsman} for spaces $W^{1,p}(D)$ with $p>d$. In
\cite{Shvartsman} the author also changes equivalently the
Euclidean metric on $D$  to a metric $\bar\varrho$ in such a way
that $u\in W^{1,p}(D)$ is uniformly continuous on $D$ with respect to
$\bar\varrho$ (since $p>d$, $u\in C(D)$). This allows him to
consider $u$ on $\partial D'$, where $D'$ is the completion of $D$
with respect to the  metric $\bar\varrho$. In our paper we
consider functions  $u\in H^1(D)$ which, in general, are not
continuous but only quasi-continuous, so the problem is more
involved. To  solve it requires us to use some notions and methods
of the potential theory.

In the second part of the paper we treat the Dirichlet  problem
for semilinear equations of the form
\begin{equation}
\label{eq1.9}
-Au=f(\cdot,u)+\mu\quad \mbox{in}\quad D,\qquad u=\psi\quad\mbox{on}\quad \partial_M D,
\end{equation}
where $\psi\in L^1(\partial_M D;h^A_m)$ and $\mu$ is a   Borel
measure on $D$ such that $\int_D\delta\,d|\mu|<\infty$. As for
$f$, we assume that it is continuous and nonincreasing with
respect to $u$. To deal with (\ref{eq1.9}), we first extend the
trace operator to the set $\TT$ ($H^1(D)\subset \TT$) of all functions $\psi\in\BB(D)$
for which there exists $g\in\BB(\partial_MD)$ such that the
process
\[
[0,\tau_D]\ni t\mapsto (\mathbf{1}_D\psi +\mathbf{1}_{\partial_MD}g)(X_t\mathbf{1}_{\{t<\tau_D\}} +X_{\tau_D-}\mathbf{1}_{\{t=\tau_D\}})
\]
is continuous at $\tau_D$ under the  measure $P_x$
for $m$-a.e. $x\in D$. For $\psi\in\TT$, we put
\[
\gamma_A(\psi):=g.
\]
Let $\|\cdot\|_{q.u.}$ denote the metric of quasi-uniform convergence.
We show that
\[
\gamma_A:(\TT,\|\cdot\|_{q.u.})\rightarrow L^0(\partial_MD;h^A_m)
\]
is continuous.

 Let $H_{\mathfrak{D}^p}$ be the space defined probabilistically as a class of harmonic functions $u$ on $D$
for which the family $\{|u|^p(X_{\tau_V}),\, V\subset\subset D\}$ is uniformly integrable under the measure $P_x$ for $m$-a.e. $x\in D$. We equipp $H_{\mathfrak{D}^p}$ with the metric induced by the norm
\[
\|u\|^p_{\mathfrak{D}^p}:=\int_D\sup_{V\subset\subset D}E_x|u(X_{\tau_V})|^p\,m(dx).
\]
For $p>1$, we define $H_{\mathfrak{S}^p}$ to be the space of all harmonic functions $u$ on $D$ such that
\[
\|u\|^p_{\mathfrak{S}^p}:=E_m\sup_{t<\tau_D}|u(X_t)|^p<\infty.
\]
We  show that  $H_{\mathfrak{D}^p}\subset \TT,\, p\ge 1$ and that
\[
\gamma_A: H_{\mathfrak{D}^p}\rightarrow L^p(\partial_M D;h^A_m)
\]
is an isometric isomorphism.  Moreover, for $p>1$,
\[
\gamma_A: H_{\mathfrak{S}^p}\rightarrow L^p(\partial_M D;h^A_m)
\]
is   a homeomorpism which implies that $H_{\mathfrak{S}^p}=H_{\mathfrak{D}^p}$. We also show that similarly to the case where $p=2$, for $p\ge 1$ we have
\[
L^p(\partial D;\omega^A_m)=L^p(\partial_M D;h^A_m)
\]
if and only if $H^1_c(D)=H^1(D)$.

In the paper we propose two different but equivalent definitions
of a solution to (\ref{eq1.9}). In both definitions by a solution  we mean a function  $u\in L^1(D;m)$ such that   $f(\cdot,u)\in L^1(E;\delta\cdot m)$.
Additionally, in the first definition we require that the equality
\[
u(x)=\int_Df(u)(y)g_D(x,y)\,m(dy)+\int_Dg_D(x,y)\,\mu(dy)+\int_{\partial_M D}\psi(y)h^A_x(dy)
\]
is satisfied for all $x\in D$. In the second definitoin,  so called Stampacchia's definition by duality,  we require that
\[
\int_D(u-\gamma^{-1}_A(\psi))G^D\eta\,dm=\int_D f(\cdot,u)G^D\eta\,dm+\int_D \widetilde{G^D\eta}\,d\mu,\quad \eta\in L^\infty(D;m),
\]
where $ \widetilde{G^D\eta}$ is a continuous $m$-version of $G^D\eta$ ($G^D$ is strongly Feller).
In general, there is no solution to (\ref{eq1.9}) (see \cite{BMP}). Following \cite{BMP1,BMP} and \cite{Kl:CVPDE}, for given $A,f$ denote by $\GG_0$ the set of all good measures for (\ref{eq1.9}), i.e the set of all bounded Borel measures $\mu$ on $D$  for which there exists a solution to (\ref{eq1.9}) with $\psi\equiv 0$.
Our main result says that for every $\mu\in\GG_0$ and $\psi\in L^1(\partial_M D;h^A_m)$
there exists a unique solution to (\ref{eq1.9}) (this agrees with the results of \cite{MarcusPonce}, where it is proved, that the class
of good measures does not depend on the boundary data in case of $A=\Delta$).  We also prove that for every $k>0$,
\begin{equation}
\label{eq1.18}
\|T_k(u)\|^2_{\breve H^1_\delta}\le 3k\lambda^{-1}(\|\psi\|_{L^1(\partial_MD;h^A_m)}
+\|f(\cdot,0)\|_{L^1(D;\delta\cdot m)}+\|\mu\|_{TV,\delta}),
\end{equation}
where $T_k(u)(x)=((-k)\vee u(x))\wedge k$. Furthermore, we show that for regular domains of class  $C^{1,1}$,
$u\in L^p(D;\delta\cdot m)$ with $p<\frac{d}{d-1}$. This, when combined with (\ref{eq1.18}), implies that $u\in W^{1,q}_\delta$ with $q<\frac{2d}{2d-1}$. Finaly, Let us note that, if $D$ is of class $C^{1,1}$, then there exist $c_1, c_2>0$ such that
\[
c_1\rho(x)\le \delta(x)\le c_2\rho(x),\quad x\in D
\]
where $\rho(x)=\mbox{dist}(x,\partial D)$. The theory of  so called very weak solutions to elliptic equations with data in $L^1(D;\rho\cdot  m)$ have attracted quite interest in recent years  (see, e.g.,  \cite{DiazRakotoson,Rakotoson} and references  therein).

\section{Weak and soft solutions to the Dirichlet problem} \label{sec2}

Let $D$ be and arbitrary bounded open subset of $\BR^d$. In this section, we provide  a stochastic representation for weak solutions of DP$(A,\partial D,\psi)$ with $\psi\in C(\partial D)$. Based on this result, we  give the definition of a soft solution to  DP$(A,\partial D,\psi)$ with $\psi\in L^2(\partial D;\omega^A_m)$, where $\omega^A_m$ is defined by (\ref{eq1.13}). We next show that, if $\psi\in C(\partial D)$, then soft solutions are weak solution to DP$(A,\partial D,\psi)$. In Section \ref{sec3}  we will  show that soft solutions of  DP$(A,\partial D,\psi)$ are solutions obtained via the Perron-Wiener-Brelot method.

In what follows, for given open bounded sets $V,U\subset\BR^d$, we write $V\subset\subset U$ if $\overline V\subset U$. By $\BB(\BR^d)$ we denote the
set of all Borel measurable functions on $\BR^d$, and by $\BB^+(\BR^d)$ the subset of $\BB(\BR^d)$ consisting of positive functions.
For a given positive Borel measure $\mu$ on $\BR^d$ and $g\in \BB^+(\BR^d)$, we denote by $g\cdot \mu$ the  Borel measure on $\BR^d$ defined by
\[
\int_{\BR^d}\eta\, d(g\cdot \mu):=\int_{\BR^d}\eta g\,d\mu,\quad \eta\in\BB^+(\BR^d).
\]
We denote by $g^\alpha_D$ the Green function for $D$ and operator $A-\alpha I$ ($\alpha\ge 0$), and by $\mbox{Cap}_A$  the capacity associated with the operator $A$. Recall that a function $u$ on $\BR^d$  is called quasi-continuous if and only if for every
$\varepsilon >0$ there exists a closed $F_\varepsilon\subset K$ such that Cap$_A(\BR^d\setminus F_\varepsilon)\le \varepsilon$ and $u_{|F_\varepsilon}$ is continuous.
For a positive Borel  measure $\mu$ on $D$ and $\alpha\ge 0$, we set
\[
R^D_\alpha\mu(x)=\int_Dg_D^\alpha(x,y)\,\mu(dy),\quad x\in D.
\]
Let $\BX$ be a diffusion on $\BR^d$ associated with the operator $A$ (see \cite{FOT,Ro:Stochastics,Stroock}). It is well known that, if $\mu=f\cdot m$ for some $f\in\BB^+(D)$, then
\begin{equation}
\label{eq2.as}
R^D_\alpha f(x):=R^D_\alpha(f\cdot m)=E_x\int_0^{\tau_D}e^{-\alpha t}f(X_t)\,dt,\quad x\in D,
\end{equation}
where $\tau_D$ is defined by (\ref{eq1.14}) and for $f\in L^p(D;m)$, $R^D_\alpha f$ is a quasi-continuous
$m$-version of $G^Df$ and if $f\in L^\infty(D;m)$ then $R^Df$ is a continuous $m$-version of $G^Df$ on $D$. Moreover, if $\mu$ is smooth in the sense of Dirichlet forms (see \cite[Section 2.2]{FOT} for the definition), then by  \cite[Lemma 2.2.10, Theorem 5.4.2]{FOT}  there exists a unique  positive continuous additive functional $A^\mu$ of $\mathbb X$ such that for q.e. $x\in D$,
\begin{equation}
\label{eq2.afr}
R^D_\alpha\mu(x)=E_x\int_0^{\tau_D}e^{-\alpha t}\,dA^\mu_t.
\end{equation}
We put $R^D:=R^D_0,\, g_D:=g^0_D$.
\begin{lm}
\label{lm2.1}
If $\psi\in C(\partial D)$, then  $E_\cdot\psi(X_{\tau_D})\in C(D)$.
\end{lm}
\begin{dow}
By Tietze's extension theorem, we may assume that  $\psi\in C_b(\BR^d)$. Let  $V\subset\BR^d$ be
a bounded open set such that $D\subset\subset V$, and let $\psi_\alpha:=\alpha R^V_\alpha \psi$. It is an elementary check  that $\psi_\alpha\rightarrow \psi$ uniformly on $\overline D$. By the strong Markov property
and (\ref{eq2.as}),
\begin{align*}
E_xe^{-\alpha\tau_D}\psi_\alpha(X_{\tau_D})&=\alpha E_xe^{-\alpha\tau_D}E_{X_{\tau_D}}\int_0^{\tau_V}e^{-\alpha r}\psi(X_r)\,dr\\&=\alpha E_xe^{-\alpha\tau_D}\int_{\tau_D}^{\tau_V}e^{-\alpha (r-\tau_D)}\psi(X_r)\,dr=\alpha E_x\int_{\tau_D}^{\tau_V}e^{-\alpha r}\psi(X_r)\,dr\\&=\alpha R^V_\alpha(x) \psi-\alpha R^D_\alpha\psi(x),\quad x\in D.
\end{align*}
Since $(R^V_\alpha)_{\alpha>0}$ and $(R^D_\alpha)_{\alpha>0}$ are strongly Feller, $E_\cdot e^{-\alpha\tau_D}\psi_\alpha(X_{\tau_D})\in C(D)$, from which the desired
assertion easily follows.
\end{dow}

\begin{lm}
\label{lm2.2}
Let $u\in H^1_0(D)$ be quasi-continuous. Then for every $x\in D$,
\[
u(X_t)\rightarrow 0\quad P_x\mbox{-a.s. as}\quad t\nearrow \tau_D.
\]
\end{lm}
\begin{dow}
Let  $\mathbb X^D$ denote the part of the process $\mathbb X$ on $D$ (see \cite[Section 4.4]{FOT}).
By Fukushima's decomposition (see \cite[Theorem 5.2.2]{FOT}), for q.e. $x\in D$,
\[
u(X^D_t)=u(x)+A_t+M_t,\quad t\ge 0,\quad P_x\mbox{-a.s.}
\]
where $A$ is a continuous additive functional of $\mathbb X^D$  of
zero energy and $M$ is a continuous martingale additive functional
of $\mathbb X^D$ of finite energy. From this we get
\[
u(X^D_t)=-(A_{\tau_D}-A_t)-(M_{\tau_D}-M_t),\quad t\le \tau_D,\quad P_x\mbox{-a.s.},
\]
since $u(X^D_{\tau_D})=u(\Delta)=0$, where $\Delta$ is an extra point which is a one-point compactification of $D$ (see \cite[Theorem A.2.10]{FOT}).
Hence we get the assertion of
the lemma for q.e. $x\in D$. Now, set
\begin{align*}
v(x)=P_x(\limsup_{t\nearrow \tau_D}|u(X_t)|>0),\quad x\in D.
\end{align*}
One can  check that $v$ is an excessive function relative to $\mathbb X^D$. Indeed, by the strong Markov property,
\begin{align*}
E_xv(X_s^D)&=E_xP_{X^D_s}(\limsup_{t\nearrow \tau_D}|u(X^D_t)|>0)=P_x(\limsup_{t\nearrow \tau_D\circ\theta_s}|u(X^D_t\circ\theta_s)|>0)\\&
=P_x(\limsup_{t\nearrow (\tau_D-s)^+}|u(X_{t+s}^D)|>0)\le P_x(\limsup_{t\nearrow \tau_D}|u(X_t)|>0)=v(x)
\end{align*}
for every $x\in D$. Since we already know that $v=0$ q.e. on $D$, we have $v=0$ on $D$ by \cite[Proposition II.3.2]{BG}.
\end{dow}

\begin{df}
\label{df.w}
Let $\psi\in C(\partial D)$. The function $B\psi$, where $B$ is defined by (\ref{eq1.15}), is called a weak solution  of DP$(A,\partial D,\psi)$.
\end{df}

In Theorem \ref{tw2.1} below, we give a probabilistic representation of a weak solution of  DP$(A,\partial D,\psi)$. Its proof is based solely upon  Fukushima's decomposition of additive functionals of $\BX$. Then, we  give another proof
based upon results of \cite{Ro:SM}, where the author used the theory of weak convergence of diffusion processes. The second proof gives us even stronger result then that formulated in Theorem \ref{tw2.1}.

\begin{tw}
\label{tw2.1}
Let $\psi\in C(\partial D)$ and let $u$ be a weak solution to \mbox{\rm(\ref{eq1.3})}. Then
\begin{equation}
\label{eq2.1}
u(x)=E_x\psi(X_{\tau_D}),\quad x\in D.
\end{equation}
\end{tw}
\begin{dow}
We first assume that  $\psi\in H^1(D)\cap C(\overline D)$.
Let $\{D_n\}$ be an increasing sequence of Lipschitz open sets such that  $D_n\subset\subset D$ and $\bigcup_{n\ge 1} D_n=D$.
Since $D_n$ is Lipschitz,  there exists an extension of $u_{|D_n}$ to $\tilde
u_n\in H^1(\BR^d)$. By Fukushima's  decomposition (see
\cite[Theorem 5.2.2]{FOT}), there
exist a unique martingale additive functional $M^n$ of finite energy and a unique continuous additive functional $A^n$ of zero energy  such that
\begin{equation}
\label{eq2.2}
\tilde u_n(X_t)=\tilde u_n(x)+A^n_t+M^n_t,
\quad t\ge 0,\quad P_x\mbox{-a.s.}
\end{equation}
for q.e. $x\in \BR^d$. By \cite[Theorem 5.4.1]{FOT},  $A^n_t=0$,
$t\le\tau_{D_n}$, so in particular,
\begin{equation}
\label{eq2.5}
u(X_{\tau_{D_n}})=u(x)+M^n_{\tau_{D_n}},\quad P_x\mbox{-a.s.}
\end{equation}
for q.e. $x\in D_n$. Taking expectation in (\ref{eq2.5}), we get
\begin{equation}
\label{eq2.5exp}
u(x)=E_xu(X_{\tau_{D_n}})
\end{equation}
for q.e. $x\in D_n$.  Since $(u-\psi)\in H^1_0(D)$, we have by Lemma \ref{lm2.2} that $(u-\psi)(X_t)\rightarrow 0$ $P_x$-a.s. as $t\nearrow \tau_D$  for q.e. $x\in D$. Since $\psi\in C(\overline D)$, $u(X_t)\rightarrow \psi(X_{\tau_D})$ $P_x$-a.s. as $t\nearrow \tau_D$  for q.e. $x\in D$. By (\ref{eq1.2.1}),  $u$ is bounded, so by the Lebesgue dominated convergence theorem, $E_xu(X_{\tau_{D_n}})\rightarrow E_x\psi(X_{\tau_D})$ for  q.e. $x\in D$. From this,  (\ref{eq2.5exp}) and Lemma \ref{lm2.1} we obtain (\ref{eq2.1}).

We now assume that $D$ is an arbitrary bounded domain and $\psi\in C(\partial D)$. Choose a sequence  $\{\psi_n\}\subset H^1(D)\cap C(\overline D)$ such that $\psi_n\rightarrow\psi$ uniformly on $\partial D$. Let $u_n$ be a weak solution to (\ref{eq1.3}) with $\psi$ replaced by $\psi_n$.
By what has already been proved,
\begin{equation}
\label{eq2.8}
u_n(x)=E_x\psi_n(X_{\tau_D}),\quad x\in D.
\end{equation}
Letting $n\rightarrow\infty$ in (\ref{eq2.8}) and using  (\ref{eq1.2.1}) yields (\ref{eq2.1}).
\end{dow}

\begin{wn}
\label{wn2.1}
For every $x\in D$,
\[
P_x(X_{\tau_D}\in dy)=\omega^A_x(dy).
\]
\end{wn}

\begin{uw}
\label{wn2.2}
Let $D$ be connected, $x\in D$ and $p\ge 1$. If $\psi\in \BB^+(\partial D)\cap L^p(\partial D; \omega_x^A)$, then $\psi\in L^p(\partial D;\omega_y^A)$
for every $y\in D$.
\end{uw}
\begin{dow}
By (\ref{eq.bb}) and Harnack's inequality for every $y\in D$ there exists $c_y>0$ such that
\begin{equation}
\label{harnack}
\omega^A_y\le c_y\omega^A_x.
\end{equation}
Hence we get the result.
\end{dow}

\begin{wn}
\label{wn2.3}
Let $D$ be connected. Assume that $\psi\in \BB^+(\partial D)\cap L^p(\partial D;\omega^A_x)$  for some $p\ge 1$, $x\in D$ and define $u$ by  \mbox{\rm(\ref{eq2.1})}. Then $u\in C(D)$ and for every $y\in D$,
\begin{equation}
\label{eq2.9}
|u(y)|\le \|\psi\|_{L^p(\partial D;\omega^A_y)}.
\end{equation}
\end{wn}
\begin{dow}
Inequality  (\ref{eq2.9}) follows immediately from (\ref{eq2.1}) and by Remark \ref{wn2.2} the right-hand
side of (\ref{eq2.9}) is finite for every $y\in D$.
Choose $r>0$ so that $B(x,r)\subset D$. By Harnack's inequality, there is $c>0$ such that $\omega^A_y\le c\omega^A_x$ for $y\in B(x,r)$. Choose
$\psi_n\in C(\partial D)$ so that
$\psi_n\rightarrow \psi$ in $L^2(\partial D,\omega^A_x)$  and define  $u_n$ by (\ref{eq2.1}) but with $\psi$ replaced by $\psi_n$. Then by (\ref{eq2.9}),
\[
|u(y)-u_n(y)|\le c   \|\psi-\psi_n\|_{L^p(\partial D;\omega^A_x)},\quad y\in B(x,r).
\]
This implies that $u\in C(D)$ because $u_n\in C(D)$ by Lemma \ref{lm2.1}.
\end{dow}
\medskip

We will see in Section \ref{sec3} that $u$ defined by  (\ref{eq2.1}) is the Perron-Wiener-Brelot solution of the problem DP$(A,\partial D,\psi)$.
Therefore the next corollary generalizes to operators of the form (\ref{eq1.1}) the corresponding result of  \cite{ArendtDaners} proved for  $A=\Delta$.

\begin{wn}
Assume that $\psi\in C(\partial D)$ and there exists $\Psi\in H^1_{loc}(D)\cap C(\overline D)$ such that
$\Psi_{|\partial D}=\psi$ and $-A\Psi\in H^{-1}(D)$. Let $u\in H^1_0(D)$ be such that
\[
-Au=A\Psi\quad\mbox{in}\quad H^{-1}(D).
\]
Then
\[
(u+\Psi)(x)=E_x\psi(X_{\tau_D}),\quad x\in D.
\]
\end{wn}
\begin{dow}
Let $\{D_n\}\subset\subset D$ be an increasing sequence such that $\bigcup_{n\ge 1} D_n$, and let $w=u+\Psi$. Then $w\in C(D)\cap H^1_{loc}(D)$ and $\EE(u,v)=0$ for $v\in C_c^2(D)$, so
$w$ is a weak solution to DP$(A,\partial D_n,w_{|\partial D_n})$. By Theorem \ref{tw2.1},
\[
w(x)=E_xw(X_{\tau_{D_n}}),\quad x\in D_n.
\]
By continuity of $\Psi$ and Lemma \ref{lm2.2}, $E_xw(X_{\tau_{D_n}})\rightarrow E_x\psi(X_{\tau_D}),\, x\in D$, which proves the corollary.
\end{dow}

Let $\bar\sigma$ denote the symmetric square root of $a$, i.e.
\begin{equation}
\label{eq2.12}
\bar\sigma(x)=a^{1/2}(x),\quad x\in \BR^d.
\end{equation}

\begin{stw}
\label{stw2.1.1}
Let $\psi\in C(\partial D)$ and let $u$ be a weak solution to \mbox{\rm(\ref{eq1.3})}. Then for every $x\in D$,
\begin{equation}
\label{eq2.333}
u(X_t)=\psi(X_{\tau_D})-\int_t^{\tau_D}\bar\sigma\nabla u(X_r)\,dB_r,\quad t\le \tau_D,\quad P_x\mbox{-a.s.}
\end{equation}
\end{stw}
\begin{dow}
We first assume that $\psi\in C(\overline D)\cap H^1(D)$. Let $\{D_n\}$ be an increasing sequence of bounded open Lipschitz subsets of $D$ such that $D_n\subset\subset D$ and
$\bigcup_{n\ge 1} D_n=D$. It is clear that $u$ is a  weak solution to DP$(A,D_n,u_{|\partial D_n})$. By \cite{Ro:SM}, there is a Wiener process $B$ such that
\begin{equation}
\label{eq2.13}
u(X_t)=u(X_{\tau_{D_n}})-\int_t^{\tau_{D_n}}\bar\sigma\nabla u(X_r)\,dB_r,\quad t\le \tau_{D_n},\quad P_x\mbox{-a.s.},\quad x\in D_n.
\end{equation}
By (\ref{eq1.2.1}),  $u$ is bounded. Therefore, by using  the Burkholder-Davis-Gundy inequality,  it may be concluded from the above equation that there exists $c>0$ depending only on $\|u\|_\infty$ such that
\[
E_x\int_0^{\tau_{D_n}}|\bar\sigma\nabla u(X_r)|^2\,dr\le c,\quad x\in D_n.
\]
Applying  Fatou's  lemma gives
\begin{equation}
\label{eq.mar}
E_x\int_0^{\tau_D}|\bar\sigma\nabla u(X_r)|^2\,dr\le c,\quad x\in D.
\end{equation}
By Lemma \ref{lm2.2} and regularity of $\psi$, we have $u(X_t)\rightarrow \psi(X_{\tau_D}),\, t\nearrow \tau_D$ $P_x$-a.s. for $x\in D$. Therefore letting $n\rightarrow\infty$ in (\ref{eq2.13}) we obtain (\ref{eq2.333}).

We now assume that $\psi\in C(\partial D)$. Choose a sequence  $\{\psi_n\}\subset C(\overline D)\cap H^1(D)$ so that $\psi_n\rightarrow \psi$ uniformly on $\partial D$. Let $u_n$ be a weak solution to (\ref{eq1.3}) with $\psi$ replaced by $\psi_n$. By what has aready been proved,
\begin{equation}
\label{eq2.333bis}
u_n(X_t)=\psi_n(X_{\tau_D})-\int_t^{\tau_D}\bar\sigma\nabla u_n(X_r)\,dB_r,\quad t\le \tau_D,\quad P_x\mbox{-a.s.}
\end{equation}
By (\ref{eq1.2.1}), $u_n\rightarrow u$ pointwise and  in $H^1_{loc}$. Moreover, by It\^o's formula,
\begin{equation}
\label{eq2.1317}
|u_n(x)-u_m(x)|^2+E_x\int_0^{\tau_D}|\bar\sigma\nabla(u_n-u_m)|^2(X_r)\,dr= E_x|\psi_n-\psi_m|^2(X_{\tau_D})
\end{equation}
for every $x\in D$. Therefore  letting $n\rightarrow \infty$ in (\ref{eq2.333bis}) and using  the Burkholder-Davis-Gundy inequality we get the desired result.
\end{dow}

\begin{uw}
Observe that by taking expectation in (\ref{eq2.333}) with $t=0$ (it is clear that the process $\int_0^\cdot\bar\sigma\nabla u(X_r)\,dB_r$ is a bounded martingale) we get (\ref{eq2.1}).
\end{uw}
\medskip

In what follows,
\[
\delta(x):= R^D1(x)=E_x\tau_D,\quad x\in D.
\]

\begin{stw}
\label{stw2.1}
Let $\psi\in L^2(\partial D;\omega^A_m)$, and let  $u$ be defined by  \mbox{\rm(\ref{eq2.1})}. Then
\begin{enumerate}
\item [\rm(i)] $u\in \breve H^1_\delta(D)\cap C(D)$, where $\breve H^1_{\delta}(D)$ is defined by \mbox{\rm(\ref{eq1.16})}.
\item [\rm(ii)] $\EE(u,v)=0$ for every $ v\in C_c^\infty(D)$.
\item[\rm(iii)] \mbox{\rm(\ref{eq1.17})}  holds true
 for every $x\in D$ and  every increasing sequence $\{D_n\}$ of open subsets of $D$ such that $\{D_n\}\subset\subset D$ and  $\bigcup_{n\ge 1} D_n=D$.

\item[\rm(iv)] $\|g^{1/2}_D(x,\cdot)\nabla u\|_{L^2(D;m)}\le \lambda^{-1}\|\psi\|_{L^2(\partial D;\omega^A_x)}$ for every $x\in D$.
\item [\rm(v)] $\|u\|_{\breve H^1_\delta}\le (1\wedge \lambda)^{-1}\|\psi\|_{L^2(\partial D;\omega^A_m)}$.

 \item[\rm(vi)] \mbox{\rm(\ref{eq2.333})} is satisfied  for every $x\in D$.

 \item [\rm(vii)] $E_x\sup_{t<\tau_D}|u(X_t)|^2\le 4\|\psi\|^2_{L^2(\partial D;\omega^A_x)}$ for every $x\in D$.
\end{enumerate}
\end{stw}
\begin{dow}
Choose  $\{\psi_n\}\subset C(\partial D)$ so that  $\psi_n\rightarrow \psi$ in $L^2(\partial D;\omega^A_m)$. Note that by Remark \ref{wn2.2}, $\psi\in L^2(\partial D;\omega^A_x)$, $x\in D$, and by (\ref{harnack}) (up to subsequence),
$\psi_n\rightarrow \psi$ in $L^2(\partial D;\omega^A_x)$ for every $x\in D$.
By  Corollary \ref{wn2.3}, $u\in C(D)$.
Let $u_n$ be given by (\ref{eq2.1}) with $\psi$ replaced by $\psi_n$. By Proposition \ref{stw2.1.1}, for all $n\ge1$ and $x\in D$ we have
\begin{equation}
\label{eq2.6}
u_n(X_t)=\psi_n(X_{\tau_D})+\int_t^{\tau_{D}}\bar\sigma\nabla u_n(X_r)\,dB_r,\quad t\in [0,\tau_{D}],\quad P_x\mbox{-a.s.}
\end{equation}
Using  It\^o's formula, we obtain (\ref{eq2.1317}), from which we deduce that $u\in \breve H^1_\delta(D)$ and $u_n\rightarrow u$ in $\breve H^1_\delta(D)$. Therefore letting $n\rightarrow\infty$ in (\ref{eq2.6})
and arguing as in the proof of Proposition \ref{stw2.1.1} we get (vi). By It\^o's formula,
\begin{equation}
\label{eq2.7}
|u(x)|^2+E_x\int_0^{\tau_{D}}|\bar\sigma\nabla u|^2(X_r)\,dr=E_x|\psi(X_{\tau_{D}})|^2,\quad x\in D.
\end{equation}
Assertions (iv) and (v) follow from (\ref{eq2.7}). As for (ii), we know from Theorem \ref{tw2.1} that it holds for $u_n$. By (v), $u_n\rightarrow u$
in $H^1_{loc}(D)$, so (ii) holds for $u$, too.  Assertion (iii) follows easily from (vi). Inequality (vii) follows from Doob's $L^2$-inequality.
\end{dow}

\begin{df}
We say that $u$ is a soft solution  of DP$(A,\partial D,\psi)$ if (i)--(iii) of Proposition \ref{stw2.1} are satisfied.
\end{df}

\begin{stw}
\label{stw2.2}
For every $\psi\in L^2(\partial D;\omega^A_m)$ there exists a unique  soft solution to \mbox{\rm DP}$(A,\partial D,\psi)$.
\end{stw}
\begin{dow}
The existence part follows from Proposition \ref{stw2.1}.
To prove uniqueness, suppose that $u_1,u_2$  are soft solutions to DP$(A,\partial D,\psi)$. Write $u=u_1-u_2$ and consider an increasing sequence $\{D_n\}$ of open subsets of $D$ such that
$\{D_n\}\subset\subset D$ and  $\bigcup_{n\ge 1} D_n=D$. Since $u\in H^1(D_n)\cap C(\overline D_n)$,
\[
u(x)=E_xu(X_{\tau_{D_n}}),\quad x\in D_n
\]
by Theorem \ref{tw2.1}. Therefore, by Corollary \ref{wn2.1} and (\ref{eq1.17}),
 the right-hand side of the above inequality converges to zero as $n\rightarrow\infty$. This proves that $u_1=u_2$.
\end{dow}

\begin{uw}
By Theorem \ref{tw2.1} and Proposition \ref{stw2.1}, the
definitions of soft solution and weak solution agree for $\psi\in
C(\partial D)$. However, be careful! When $D$ is a Lipschitz
domain, then one can define a weak solution to DP$(A,\partial
D,\psi)$ for $\psi\in H^{1/2}(\partial D)$ as a function $u\in
H^1(D)$ such that (\ref{eq1.2.0}) is satisfied and Tr$(u)=\psi$,
where Tr is the usual trace operator. In general, $u$ defined in
this way does not have the property formulated in Proposition
\ref{stw2.1}(iii). The reason is  that the trace is defined
$\sigma$-a.e. ($\sigma$ is the Lebesgue surface measure on
$\partial D$), so the weak solution to DP$(A,D,\psi)$ defined via
the operator Tr does not depend on the $\sigma$-version of $\psi$.
This is not true for soft solutions because they are defined by
(\ref{eq2.1}), and in general, the  measure $\omega^A_m$ is not
absolutely continuous with respect to $\sigma$ (in fact, it may be
completely singular, see \cite{CFK,ModicaMortola}). In different
words, if $\psi_1=\psi_2$ $\sigma$-a.e., then weak solutions of
DP$(A,\partial D,\psi_i)$ defined via the trace operator  Tr are
equal, but it may happen that $u_1(x)\neq u_2(x)$, $x\in D$, where
$u_1, u_2$ are defined by (\ref{eq2.1}) with $\psi$ replaced by
$\psi_1$ and $\psi_2$, respectively. The definition of soft
solution is just more sensitive to the boundary values.
In Section \ref{sec4} we will
show that, if $u$ is a weak solution to DP$(A,\partial D,\psi)$
defined via  the trace operator Tr, then it is a soft solution of
DP,  but with some specially chosen  $\sigma$-version of $\psi$
(defined $\omega^A_m$-a.e.).
\end{uw}

\section{Perron-Wiener-Brelot solutions} \label{sec3}

We begin with recalling some  notions from \cite{BH}. Let $E$ be a locally compact
separable metric space.
For a given class of functions $F\subset \BB^+(E)$, we set
$S(F)=\{\sup_{n\ge1} f_n:\{f_n,n\ge 1\}\subset F\}$. We say that
$F$ is $\sigma$-stable if $S(F)=F$. Let $\mathsf{W}$ be a convex
cone of positive numerical lower semi-continuous (l.s.c) functions
on $E$. By $\mathcal{T}^f$ we denote the smallest topology under
which all functions from $\mathsf{W}$ are continuous (so called
fine topology). For  $u\in\BB(E)$, we denote  by $\hat{u}^f$  its
l.s.c. regularization with respect to the topology
$\mathcal{T}^f$.

We say that a pair  $(E,\mathsf{W})$ is the balayage space if
\begin{enumerate}
\item[$(B_1)$] $\mathsf{W}$ is $\sigma$-stable,
\item[$(B_2)$]  for every  $\mathsf{V}\subset \mathsf{W}$, $\widehat{\inf \mathsf{V}}^f\in\mathsf{W}$,
\item[$(B_3)$] for all $u,v,w\in\mathsf{W}$ such that $u\le v+w$ there exist $\bar{v},\bar{w}\in\mathsf{W}$
such that $u=\bar{v}+\bar{w},\, \bar{v}\le v,\, \bar{w}\le w$,
\item[$(B_4)$]  there exists a function cone $\mathsf{P}\subset C^+(E)$ such that $S(\mathsf{P})=\mathsf{W}$.
\end{enumerate}

\begin{prz}
\label{prz3.1}
Let $\{P_t,\, t\ge 0\}$ be the semigroup generated by the operator (\ref{eq1.1}). By \cite[Corollary II.4.9]{BH}, the pair $(E_\mathbb{P},\BR^d)$, where $E_\mathbb{P}$ is
the set of excessive functions on $\BR^d$  defined as $E_\mathbb{P}=\{f\in\BB^+(\BR^d):\sup_{t>0}P_tf=f\}$, is a balayage space.
\end{prz}

It is well known  (see \cite[Section III.2]{BH}) that every balayage space $(E,\mathsf{W})$ generates naturally
the so called harmonic kernel $\{H_U,\, U\subset E, U-\mbox{open}\}$. By $^*H(U)$ we denote the set of all hyperharmonic functions on $U$, i.e.
\[
^*H(U)=\{u\in\BB(E);\, u_{|U}\,\,\mbox{is l.s.c.},\,\,-\infty <H_V(u)(x)\le u(x),\, x\in V,\, V-\mbox{open},\, \overline{V}\subset U\}.
\]
The class $H(U)$  of harmonic functions is defined by
\begin{equation}
\label{eq3.6} H(U)=\,^*H(U)\cap(-^*H(U)).
\end{equation}
For $\psi\in \BB^+(E)$,  we set
\[
\mathcal{U}^U_\psi=\{u\in  {^{*}H(U)}:u\mbox{ is lower bounded on
}  U\mbox {and }\liminf_{y\rightarrow x} u(y)\ge\psi(x),\,
x\in\partial U\},
\]
\[
 \mathcal L^{U}_\psi=-\mathcal U^{U}_{-\psi},
\]
and we define operators $\overline{H}^U$ and $\underline{H}^U$ by
\[
\overline{H}^U\psi\equiv \inf \mathcal U^U_\psi,\quad \underline{H}^U \psi\equiv\sup \mathcal L^U_\psi.
\]
If $\overline{H}^U\psi= \underline{H}^U \psi$, we set $H^U\psi= \underline{H}^U \psi$.

\subsection{Linear equations}

From now on we consider the  balayage space from  Example
\ref{prz3.1}. As in Section \ref{sec2}, $D$ is an arbitrary
bounded open subset of $\BR^d$. By \cite[Theorem VI.3.16]{BH}, for
all bounded open $V\subset\BR^d$ and $u\in\BB^+(\BR^d)$,
\begin{equation}
\label{eq3.1}
H_V(u)(x)=E_xu(X_{\tau_V}),\quad x\in V.
\end{equation}

\begin{df}
Let $\psi\in\BB(\partial D)$. We say that $u:D\rightarrow\BR$ is a Perron-Wiener-Brelot solution of DP$(A,\partial D,\psi)$
if $u=\overline{H}^D\psi=\underline{H}^D \psi$.
\end{df}

\begin{stw}
\label{stw3.1}
Assume that $u\in\BB(D)$ is lower bounded and finite $m$-a.e. Then $u\in\, ^*H(D)$ if and only if  there exists a  positive Borel measure $\mu$ on $D$ such that
\begin{equation}
\label{eq3.2}
u-R^D\mu\in H^1_{loc}(D),\qquad \EE(u-R^D\mu,v)=0,\quad v\in C_c^1(D).
\end{equation}
\end{stw}
\begin{dow}
Assume that $u\in\,^*H(D)$. Since $u$ is lower bounded, we may assume that $u$ is positive.
By \cite[Corollary II.5.3]{BG}, $u$ is an excessive function with respect to the resolvent $(R^D_\alpha)_{\alpha\ge0}$. Therefore,
by  Riesz's decomposition theorem (see \cite{GetoorGlover}), there exists a Borel positive measure $\mu$ on $D$ and a positive harmonic function $h$ such that
\[
u=R^D\mu+h.
\]
By the above, the definition of a harmonic function and  (\ref{eq3.1}),
\[
E_x(u-R\mu)(X_{\tau_V})=u(x)-R\mu(x),\quad x\in V
\]
for any $V\subset\subset D$. By Corollary \ref{wn2.3}, $u-R\mu\in C(D)$ and by Proposition \ref{stw2.1} (\ref{eq3.2}) holds.
Now assume that (\ref{eq3.2}) is satisfied. By \cite{Nash}, $u-R\mu\in C(D)$, so by Theorem \ref{tw2.1} and (\ref{eq3.1}),
$u\in\,^*H(D)$.
\end{dow}

\begin{wn}
Function $u\in \BB(D)$ is harmonic iff $u\in H^1_{loc}(D)$ and
\[
\EE(u,v)=0,\quad v\in C_c^1(D).
\]
\end{wn}

\begin{tw}
\label{tw3.p} Assume that $\psi\in L^1(\partial D;\omega^A_m)$.
Then there exists a unique PWB-solution $u$ to \mbox{\rm DP}$(A,\partial
D,\psi)$. Moreover,
\[
u(x)=E_x\psi(X_{\tau_D}),\quad x\in D.
\]
\end{tw}
\begin{dow}
By Corollary \ref{wn2.2}, $\psi\in L^1(\partial D;\omega_x^A)$ for
every $x\in D$. Therefore, by \cite[Corollary VII.2.12]{BH}, there
exists a PWB-solution $u$ to DP$(A,\partial D,\psi)$ and
$u=H_D\psi$. By virtue of (\ref{eq3.1}), this proves the theorem.
\end{dow}
\medskip

The following corollary follows immediately from  Proposition
\ref{stw2.2} and Theorem \ref{tw3.p}.

\begin{wn}
\label{wn3.p} If   $\psi\in L^2(\partial D;\omega^A_m)$, then $u$
is a  soft solution to \mbox{\rm DP}$(A,\partial D,\psi)$ if and only if it
is a PWB-solution. If $\psi\in C(\partial D)$, then $u$ is a weak
solutions to \mbox{\rm DP} $(A,\partial D,\psi)$ if and only if it is
PWB-solution.
\end{wn}

Let us recall that a bounded open set is called  regular at
$x\in\partial D$ for $A$ if each weak solution $u$ to
DP$(A,\partial D,\psi)$ with $\psi\in C(\partial D)$ has the
property that $u(y)\rightarrow \psi(x)$ as $y\rightarrow x,\,y\in
D$. The following corollary is a consequence of Corollary
\ref{wn3.p} and \cite[Theorem 1.23]{ChungZhao}. Note that
\cite[Theorem 1.23]{ChungZhao} is proved for Brownian motion, but
its proof for the diffusion $\mathbb X$ goes without any changes
(the inequality preceding \cite[Eq. (35)]{ChungZhao} holds true
for $\mathbb X$ by \cite[Lemma  II.1.2]{Stroock}).

\begin{wn}
\label{wn3.6}
The following statements are equivalent:
\begin{enumerate}
\item [\rm(i)] $D$ is regular at $x\in\partial D$ for $A$.
\item [\rm(ii)] $P_x(\tau_D>0)=0$.
\end{enumerate}
\end{wn}

The following result was proved by different  methods in
\cite{LSW}.

\begin{wn}
Let $D$ be a bounded  open  subset of $\BR^d$ and let $x\in
\partial D$. Then $x$ is regular for $A$ if and only if it  is
regular for $\Delta$.
\end{wn}
\begin{dow}
Follows from Corollary \ref{wn3.6} and \cite[Corollary 1]{Kanda}.
\end{dow}

\begin{uw}
By \cite[Remark 1]{Kanda},  if the fine topologies  $\mathcal
O_1,\mathcal O_2$ generated by two operators $A_1, A_2$  are
equal, then the regular points for $A_1$ and $A_2$ are the same for
each open set $D\subset\BR^d$. Since the fine topology is
generated by excessive functions, the well known estimates for
Green functions (see, e.g., \cite{LSW}) imply that the fine
topologies generated by $A$ and $\Delta$ are the same.
\end{uw}

In Theorem \ref{tw3.p} we obtained stochastic representation
of  PWB-solutions by using the  theory of balayage spaces developed in \cite{BH}. This theory is rather abstract and advanced.  In the next subsection we derive the stochastic representation (in more general context of semilinear equation) in more elementary way, without referring  to the results of \cite{BH}.

\subsection{Semilinear equations}

From now on, we treat formula (\ref{eq3.1}) as the  definition of
the operator $H_V$. If $\mathbb X$ is a Wiener process (i.e.
$A=\frac12\Delta$), it is rotationally invariant. Using this
property, one can show the following formula
\[
H_{B(x,r)}(u)(x) =\frac{1}{\sigma(\partial B(x,r))}\int_{\partial
B(x,r)}u(y)\,d\sigma(y)
\]
(see \cite[Proposition 1.21]{ChungZhao}).
\begin{uw}
It is worth mentioning that  the  proof of Proposition \ref{stw3.1} is based on  formula (\ref{eq3.1}),  and that in the  proof we do not use the fact that  $\{H_V\}$ is a harmonic kernel. Therefore in Proposition \ref{stw3.1} we have proved without referring  to the theory of balayage spaces that a lower bounded finite $m$-a.e. $u\in\BB(D)$  is hyperharmonic if and only if  (\ref{eq3.2}) is satisfied.
\end{uw}
We will consider the
following equation
\begin{equation}
\label{eq3.33} -Au=f(\cdot,u)+\mu\quad \mbox{on}\quad D, \qquad
u=\psi\quad\mbox{on}\quad \partial D,
\end{equation}
where $\psi\in L^1(\partial D;\omega^A_m)$ and $\mu$ is smooth
measure (i.e. absolutely continuous with  respect to Cap$_A$) such
that $R^D|\mu|<\infty$ q.e.

\begin{df}
We say that $u$ is the PWB-solution  of (\ref{eq3.33}) if
$R^D|f(\cdot,u)|<\infty$ q.e. and
\[
u-R^D(f(\cdot,u))-R^D\mu=H^D\psi\quad\mbox{q.e.}
\]
\end{df}

\begin{df}
\label{df.sub} We say that a Borel measurable function $u$ on $D$
is a subsolution (resp. supersolution) of (\ref{eq3.33}) if
$u-R^D(f(\cdot,u))-R^D\mu\in  \mathcal L^D_\psi $ (resp.
$u-R^D(f(\cdot,u))-R^D\mu\in  \mathcal U^D_\psi $).
\end{df}

We will see that the solution of (\ref{eq3.33}) exists,  and
moreover, it is the supremum over all supersolutions  of
(\ref{eq3.33}). To prove this, we will need the following two
lemmas.

\begin{lm}
\label{lm.kato} Assume that $u_i-R^D\mu_i\in \mathcal
L^D_{\psi_i},\, i=1,2$, for some smooth measures $\mu_1,\mu_2$
such that $R^D|\mu_i|<\infty$ q.e. Then
\[
u_1\vee u_2-R^D(\mathbf{1}_{\{u_1>u_2\}}\cdot \mu_1)
-R^D(\mathbf{1}_{\{u_1\le u_2\}}\cdot\mu_2)
\in \mathcal L^D_{\psi_1\vee\psi_2}.
\]
\end{lm}
\begin{dow}
Write $h_i=u_i-R^D\mu_i$. By \cite[Corollary II.5.3]{BG} and
\cite[Excercise 4.20]{Sharpe}, $-h_i(X)$ is a c\`adl\`ag
supermartingale under $P_x$ for q.e. $x\in D$. Therefore, for q.e.
$x\in D$, there is  an increasing predictable c\`adl\`ag process
$A^{x,i}$ and a local martingale $M^{x,i}$ such that
\[
h_i(X_t)=h_i(X_0)+\int_0^t\, dA^{x,i}_r
+\int_0^t\,dM^{x,i}_r,\quad t<\tau_D\quad P_x\mbox{-a.s.}
\]
Hence
\[
u_i(X_t)=u_i(X_0)-\int_0^t\, dA^{\mu_i}_r +\int_0^t\,
dA^{x,i}_r+\int_0^t\, d\tilde{M}^{x,i}_r,\quad t<\tau_D \quad
P_x\mbox{-a.s.}
\]
for some  local martingale $\tilde{M}^{x,i}$. Therefore, applying
the It\^o-Meyer formula, we see that for q.e. $x\in D$,
\begin{align*}
(u_1\vee u_2)(X_t)&=(u_1\vee u_2)(X_0)
-\int_0^t\mathbf{1}_{\{u_1>u_2\}}(X_{r-})\,dA^{\mu_1}_r\\
&\quad -\int_0^t\mathbf{1}_{\{u_1\le u_2\}}(X_{r-})\,dA^{\mu_2}_r
+\int_0^t\mathbf{1}_{\{u_1>u_2\}}(X_{r-})\,dA^{x,1}_r\\
&\quad+\int_0^t\mathbf{1}_{\{u_1\le u_2\}} (X_{r-})\,dA^{x,2}_r
+\int_0^t\,dL^x_r+\int_0^t\,dN^x_r,\quad t<\tau_D,\quad
P_x\mbox{-a.s.}
\end{align*}
for some  local martingale $N^x$ and an increasing c\`adl\`ag
process $L^x$. This implies the desired result.
\end{dow}
\medskip

Let us consider the following assumptions:
\begin{enumerate}
\item[(H1)] $E_x|\psi(X_{\tau_D})|<\infty$ for q.e. $x\in D$.

\item[(H2)] There exists $g\in \BB^+(D)$ such that $R^Dg<\infty$
and $ f(x,y)y\le g(x)|y|$ for all $x\in D$, $y\in\BR$.

\item[(H3)] $y\mapsto f(x,y)$ is continuous for every $x\in D$.

\item[(H4)] If $\underline u, \overline u\in\BB(D),\,
\underline u\le\overline u$ and $R^D|f(\cdot,\underline
u)|+R^D|f(\cdot,\overline u)|<\infty$ q.e., then there exists
$h\in \BB^+(D)$ such that $R^Dh<\infty$ q.e. and $|f(\cdot,u)|\le
h$ for every $u\in\BB(D)$ with $\underline u\le u\le \overline u$.
\end{enumerate}

\begin{lm}
\label{lm.scomp} Assume that $\underline{u}$ (resp.
$\overline{u}$)  is a subsolution (resp. supersolution) of
\mbox{\rm(\ref{eq3.33})} such that $\underline u\le\overline u$.
Let
\begin{equation}
\label{eq64.3} \bar f(x,y)= \left\{
\begin{array}{ll}f(x,\overline{u}(x))
,\quad &y\ge \overline{u}(x),
\smallskip\\
 \,\,\,f(x,y),&\underline{u}(x)<y<\overline{u}(x),
 \smallskip\\
  f(x,\underline{u}(x)),& y\le\underline{u}(x),
\end{array}
\right.
\end{equation}
and let $u$ be a solution of \mbox{\rm(\ref{eq3.33})} with $f$
replaced by $\bar f$. Then $\underline{u}\le u\le \overline{u}$.
\end{lm}
\begin{dow}
By the definitions of a  solution and a supersolution of
(\ref{eq3.33}),
\[
(u-\underline{u})-(R^D\bar f_u-R^D f_{\underline{u}})\in\mathcal
L^D_0.
\]
By Lemma \ref{lm.kato},
\[
(u-\overline{u})^+
-R^D(\mathbf{1}_{\{u>\overline{u}\}}
(\bar f_u-f_{\overline{u}}))\in\mathcal L^D_0.
\]
But $\mathbf{1}_{\{u>\overline{u}\}}(\bar f_u
-f_{\overline{u}})=0$, which implies that  $(u-\overline
u)^+\in\mathcal L^D_0$. Hence  $(u-\overline u)^+=0$,
so $u\le\overline u$. Similarly we prove that $\underline u\le u$.
\end{dow}

\begin{wn}
\label{wn.per} Assume that \mbox{\rm(H1)--(H4)} are satisfied and
there exists a subsolution $\underline{u}$ and a supersolution
$\overline{u}$ of \mbox{\rm(\ref{eq3.33})} such that $\underline
u\le \overline u$. Then there exists a solution $u$ of
\mbox{\rm(\ref{eq3.33})} such that $\underline u\le u\le \overline
u$.
\end{wn}
\begin{dow}
Observe that $u$ is a solution to (\ref{eq3.33})  if and only if
$w=u-H_D\psi$ is a solution to the problem
\[
-Aw=f_{H_D\psi}(\cdot,w)\quad\mbox{in}\quad D,
\qquad w_{|\partial D}=0
\]
with
\[
f_{H_D\psi}(x,y)=f(x,y+H_D\psi(x)),\quad x\in D,\, y\in\BR.
\]
Therefore, by \cite[Theorem 3.4]{Kl:PA2} there exists a solution $u$ of
(\ref{eq3.33}) with $f$ replaced by $\bar f$, where $\bar f$ is
defined by (\ref{eq64.3}). By Lemma \ref{lm.scomp}, $\underline
u\le u\le \overline u$, so in fact  $u$ is a solution of
(\ref{eq3.33}).
\end{dow}

\begin{stw}
Assume that \mbox{\rm(H1)--(H4)} are satisfied and   there exists a subsolution $\underline{u}$ and a supersolution  $\overline{u}$ of \mbox{\rm(\ref{eq3.33})} such that $\underline u\le \overline u$. Then
\begin{equation}
\label{eq.sup}
w= \sup\{v\le \underline u:v\mbox{ is a subsolution of } (\ref{eq3.33})\}
\end{equation}
is a solution of \mbox{\rm(\ref{eq3.33})}.
\end{stw}
\begin{dow}
Set $\mathcal C=\{v\le \underline u:v\mbox{ is a subsolution of } (\ref{eq3.33})\}$. By the assumptions of the proposition, $\mathcal C$
is nonempty. By \cite[Theorem V.(1.17)]{BG} and Lemma \ref{lm.kato}, there exists a nondecreasing sequence $\{w_n\}\subset \mathcal{C}$
such that $w_n\nearrow w$ q.e. It is clear by assumptions (H3), (H4) that
\[
H_V(w_n-R^Df_{w_n}-R^D\mu)\rightarrow H_V(w-R^Df_{w}-R^D\mu)
\]
for every open set $V\subset\subset D$. Hence $w\in\mathcal C$.  By
Corollary \ref{wn.per} there exists a solution $u$ of
(\ref{eq3.33}) such that $w\le u\le \overline u$. But
$u\in\mathcal C$, so $u\le w$, which implies that $u=w$.
\end{dow}

\section{Weak Dirichlet problem vs Dirichlet
problem on Lipschitz domain}
\label{sec4}

Throughout this section, we assume that $D$ is a bounded  open
Lipschitz  subset of $\BR^d$. It is well known  (see, e.g.,
\cite{Gagliardo,Grisvard}) that then exists a linear
continuous operator (trace operator)
\[
{\rm Tr}: H^1(D)\rightarrow L^2(\partial D; \sigma)
\]
such that
\[
{\rm Tr} (u)=u_{|\partial D},\quad u\in C(\overline D)\cap H^1(D).
\]

\begin{df}
\label{df4.1}
Let $\psi\in \gamma (H^1(D))=H^{1/2}(\partial D)$. We say that
$u\in H^1(D)$ is a weak solution of DP$(A,\partial D,\psi)$ if
(\ref{eq1.2.0}) is satisfied and  ${\rm Tr}(u)=\psi$.
\end{df}

\begin{uw}
It is clear that if $\psi\in H^{1/2}(\partial D)\cap C(\partial D)$, then
$u$ a weak solution of DP$(A,\partial D,\psi)$ in the sense of Definition \ref{df4.1} if and only if it is a weak solution in the sense of Definition \ref{df.w}.
\end{uw}

The aim of this section is to explain the relation between weak
solutions of DP and PWB-solutions of DP. It is known that for
every $v\in H^1(D)$,  ${\rm Tr}(v)=0$ if and only if $v\in
H^1_0(D)$. Consequently, if  $u$ is a weak solution to
DP$(A,\partial D,\psi)$, then $u$ is a solution to
wDP$(A,D,\bar\psi)$ with $\bar\psi\in H^1(D)$  such that
${\rm Tr}(\bar\psi)=\psi$. Therefore the results of this section are
the first step in describing a relation between solutions of wDP
and DP.

Let us recall from the previous section, that each PWB-solution of
DP$(A,\partial D,\psi)$ is of the form
\[
v(x)=\int_{\partial D} \psi(y)\,\omega ^A_x(dy)
\]
for some $\psi\in \BB(\partial D)$. It is well  known (see
\cite{CFK}) that the measure $\omega^A_m$ may be completely
singular with respect to the surface measure $\sigma$ on $\partial D$.
As a consequence,   PWB-solutions depend on the  $\sigma$-version of
$\psi$. On the other hand, weak solutions of DP are not sensitive to the $\sigma$-version of $\psi$. For this reason the relation between weak and
PWB-solutions to DP is a rather delicate matter.

\begin{stw}
\label{stw4.1} Let $\psi\in H^{1/2}(\partial D)$ and let $u$ be a weak
solution of \mbox{\rm DP}$(A,\partial D,\psi)$.  Then for $m$-a.e. $x\in D$,
\[
u(x)=\int_{\partial D}\tilde\psi (y)\,\omega^A_x(dy),
\]
where  $\tilde{\psi}$ is a quasi-continuous $m$-version  of
an extension $\bar \psi\in H^1(\BR^d)$ of $\psi$.
\end{stw}
\begin{dow}
Let $\{D_n\}$ be an increasing sequence of  open sets such that
$D_n\subset\subset D$ and $\bigcup_{n\ge 1} D_n=D$. Since $u$ is a
weak a solution to DP$(A,\partial D,\psi)$, $u\in H^1(D)\cap
C(D)$. Therefore, $u$ is a weak solution to DP($A,\partial
D_n,u_{|\partial D_n}$) for each $n\ge 1$. Consequently,  by
Theorem \ref{tw2.1},
\begin{equation}
\label{eq4.1}
u(x)=E_{x}u(X_{\tau_{D_n}}),\quad x\in D.
\end{equation}
By the definition of a weak solution, ${\rm Tr}(u)=\psi$. Hence
$u-\bar\psi\in H^1_0(D)$, so by Lemma \ref{lm2.2}, for q.e. $x\in
D$, $ (u-\tilde\psi)(X_t)\rightarrow 0$ $P_x$-a.s. as $t\nearrow
\tau_D$. Since $\tilde\psi$ is quasi-continuous,
$\tilde\psi(X_t)\rightarrow \tilde\psi(X_{\tau_D})$ $P_x$-a.s. as
$t\nearrow \tau_D$. Hence, for q.e. $x\in D$, $ u(X_t)\rightarrow
\tilde\psi(X_{\tau_D})$ $P_x$-a.s. as $t\nearrow \tau_D$.
By \cite[Remark 2.13]{Kl:PA2}, the family $\{u(X_\tau), \tau\in\TT\}$ is uniformly integrable under the measure $P_x$ for q.e. $x\in D$. Hence, by  the Vitali convergence theorem, $E_xu(X_{\tau_{D_n}})\rightarrow E_x\tilde\psi(X_{\tau_D})$ q.e., which when combined with (\ref{eq4.1}) and Corollary \ref{wn2.1} gives the desired result.
\end{dow}

\begin{wn}
\label{wn4.3} Let  $\psi\in H^{1/2}(\partial D)$ and let
$\tilde\psi$ be as in  Proposition \ref{stw4.1}. If $u$ is a
weak solution to \mbox{\rm DP}$(A,\partial D,\psi)$, then $u$ is a soft solution and PWB-solution to \mbox{\rm DP}$(A,\partial D,\tilde\psi_{|\partial D})$.
\end{wn}

\begin{uw}
In general,  the function $\tilde\psi_{|\partial D}$ appearing in
Corollary \ref{wn4.3} is not a $\sigma$-version of ${\rm Tr}(\psi)$.
In general, the measures $\sigma$ and $\omega^A_m$  determine different
equivalent classes (no inclusion between equivalent classes). However, there exists an $\omega^A_m$-version
of $\tilde\psi_{|\partial D}$ which is a $\sigma$-version of
${\rm Tr}(\psi)$. The construction of such a version is as follows.
Let $\varphi_\alpha:=\alpha R_\alpha(T_\alpha(\tilde{\psi}))$.
Then $\varphi_\alpha\in H^1(\BR^d)\cap C(\BR^d)$, so by the
definition of the trace operator,
\[
{\rm Tr}( \varphi_\alpha)=(\varphi_\alpha)_{|\partial D}.
\]
It is known that $\varphi_\alpha\rightarrow \tilde\psi$  in
$H^1(\BR^d)$. Hence ${\rm Tr}(\varphi_\alpha)\rightarrow
{\rm Tr}(\tilde{\psi})$ in $L^2(\partial D;\sigma)$ by continuity of
the trace operator. Moreover, by \cite[Theorem 2.1.4]{FOT}, we may
assume that $\varphi_\alpha\rightarrow \tilde \psi$ q.e.
Therefore, by Corollary \ref{wn2.1}, $(\varphi_\alpha)_{|\partial
D}\rightarrow \tilde \psi_{|\partial D}$ $\omega^A_m$-a.e. It
follows that $g $ defined as  $g=\limsup_{\alpha\rightarrow
\infty}(\varphi_\alpha)_{|\partial D}$ has the property that
$g={\rm Tr}(\psi)$ $\sigma$-a.e. and $g=\tilde\psi_{|\partial D}$\,
$\omega^A_m$-a.e.
\end{uw}

\section{Trace operator on $H^1(D)$ for arbitrary bounded open set $D$} \label{sec5}

Let $\varrho$ be the metric defined by (\ref{eq1.6}),  $D^*$ be
the completion of $D$ with respect to $\varrho$,  and let
$\partial_M D=D^*-D$. By \cite[Proposition 14.6]{ChungWalsh}, for
every $x\in D$,
\[
P_x(X_{\tau_D-}:=\lim_{t\nearrow \tau_D}X_t\mbox{ exists with respect to  the
metric }\varrho)=1.
\]
From now on $X_{\tau_D-}$ denotes $\lim_{t\nearrow \tau_D}X_t$, where limit is taken in metric $\varrho$.
It is clear that for every $x\in D$,
\[
P_x(X_{\tau_D-}\in\partial_M D)=1.
\]
Therefore, we can define a measure $h^A_x$ on $\partial_M D$
(called harmonic measure) as
\begin{equation}
\label{eq5.mhm}
h^A_x(dy)=P_x(X_{\tau_D-}\in dy).
\end{equation}
Let $\mathbb X^D=(X^D,\{P_x,\, x\in D\cup\{\Delta\}\},\FF,\tau_D)$ denote the process $\mathbb X$ killed upon
leaving $D$ (see, e.g., \cite{BG,FOT} for the definition of the
killed process). Let us recall that $\{\Delta\}$ is one-point compactification of $D$ and by convention $u(\Delta)=0$. By $\mathcal{I}$ we denote the invariant
$\sigma$-field for $\mathbb X^D$, i.e. $A\in \mathcal{I}$ if and
only if $A\in \FF$ and for every $t\ge0$,
\[
\theta^{-1}_t(A\cap\{\tau_D>0\})=A\cap\{\tau_D>t\}.
\]
Let
\[
X^*_t=\mathbf{1}_{[0,\tau_D)}(t)X_t
+\mathbf{1}_{\{t=\tau_D\}}X_{\tau_D-}\,,\quad t\in [0,\tau_D].
\]
By $\gamma_A$ we denote the  linear operator
\[
\gamma_A:H^1(D)\rightarrow L^2(\partial_MD;h^A_m)
\]
which  to every $\psi\in H^1(D)$  assigns the unique function
$\gamma_A(\psi)\in L^2(\partial_MD;h^A_m)$ such that the process
\begin{equation}
\label{eq5.tr0}
[0,\tau_D]\ni t\rightarrow (\mathbf{1}_D\psi
+\gamma_A(\psi)\mathbf{1}_{\partial_MD})(X^*_t)
\end{equation}
is continuous under the measure $P_x$ for $m$-a.e.  $x\in D$. Uniqueness of the trace $\gamma_A(\psi)$ follows from the fact that by the definition of $\gamma_A$,
\begin{equation}
\label{eq.tr1}
\gamma_A(\psi)(X_{\tau_D-})=\lim_{t\nearrow\tau_D}\psi(X_t)\quad P_x\mbox{-a.s.}
\end{equation}

\begin{uw}
\label{uw.sta}
Let $u_n$, $n\ge1$, be a  solution
to wDP$(A,D,\psi_n)$ and $u$ be a solution to wDP$(A,D,\psi)$. By an elementary calculus, if $\psi_n\rightarrow \psi$ in $H^1(D)$, then $u_n\rightarrow u$ in $H^1(D)$.
\end{uw}

\begin{tw}
\label{tw5.1}
The operator $\gamma_A$ is well defined and continuous.
\end{tw}
\begin{dow}
Let $\psi\in H^1(D)$. Then, by \cite{LSW}, there exists a unique
solution $u\in H^1(D)\cap C(D)$ of wDP$(A,D,\psi)$. Let $\{D_n\}$
be an increasing sequence of open sets such that
$D_n\subset\subset D$ and $\bigcup_{n\ge 1} D_n=D$. It is clear that
$u$ is a weak solution to  DP$(A,D_n,u_{|\partial D_n})$ for
$n\ge1$. Consequently, by Theorem \ref{tw2.1},
\begin{equation}
\label{eq5.1}
u(x)=E_xu(X_{\tau_{D_n}}),\quad x\in D.
\end{equation}
By Proposition \ref{stw2.1.1}, for every $x\in D_n$,
\begin{equation}
\label{eq5.2}
u(X_t)=u(X_{\tau_{D_n}})
-\int_t^{\tau_{D_n}}\bar\sigma\nabla u(X_{r})\,dr,
\quad t\in [0,\tau_{D_n}],
\end{equation}
where $\bar\sigma$ is defined by (\ref{eq2.12}). It follows that $u(X)$ is a martingale on $[0,\tau_{D_n}]$ for every $ n\ge1$, and hence  a martingale on $[0,\tau_D)$. By It\^o's formula,
\begin{align}
\label{eq5.3} E_x|u(X_{\tau_{D_n}})|^2&=|u(x)|^2
+E_x\int_0^{\tau_{D_n}}|\bar\sigma\nabla u(X_r)|^2\,dr\nonumber\\
&\le |u(x)|^2 +\|a\|_{\infty}E_x\int_0^{\tau_{D}}|\nabla
u(X_r)|^2\,dr.
\end{align}
Since $u\in C(D)$, the left-hand side of the above  inequality is
finite for $m$-a.e. $x\in D$. The integral on the right-hand side
is equal to $R^D|\nabla u|^2$. Since $R^D|\nabla u|^2\in L^1(D;m)$,
the  process $u(X)$ is a uniformly integrable martingale
on $[0,\tau_D)$ under the measure $P_x$ for $m$-a.e. $x\in D$. Let
\[
Y=\limsup_{t\nearrow \tau_D} u(X_t).
\]
Then by the martingale convergence theorem, for $m$-a.e $x\in D$,
$E_x|Y|<\infty$, $Y=\lim_{t\nearrow \tau_D} u(X_t)$  $P_x$-a.s.
and in $L^1(\Omega,P_x)$. It is an elementary check that
$Y\in\mathcal{I}$. Hence, by \cite[Theorem 14.10]{ChungWalsh},
there exists $g\in\BB(\partial_MD)$ such that for $m$-a.e. $x\in
D$,
\begin{equation}
\label{eq.tr} Y=g(X_{\tau_D-})\quad P_x\mbox{-a.s.}
\end{equation}
Using Fatou's lemma, we deduce from (\ref{eq5.3}) and (\ref{eq.tr})  that
\begin{equation}
\label{eq5.4} \|g\|^2_{L^2(\partial_MD;h^A_m)} \le
\|u\|^2_{L^2(D;m)}+\|a\|_\infty\|\delta\|_\infty\|\nabla
u\|^2_{L^2(D;m)}.
\end{equation}
It is clear that the process
$(u\mathbf{1}_D+g\mathbf{1}_{\partial_M D})(X^*)$ is continuous on
$[0,\tau_D]$. By the definition of a solution to wDP$(A,D,\psi)$,
$u-\psi\in H^1_0(D)$, so by Lemma \ref{lm2.2},
$(u-\psi)(X_t)\rightarrow 0$ $ P_x$-a.s. as $t\nearrow \tau_D$ for
$m$-a.e. $x\in D$. Hence $\psi(X_t)\rightarrow g(X_{\tau_D-})$ $P_x$-a.s. as
$t\nearrow \tau_D$ for $m$-a.e. $x\in D$. This implies
that the process
$(\mathbf{1}_D\psi+g\mathbf{1}_{\partial_MD})(X^*)$ is continuos
on $[0,\tau_D]$ under the  measure $P_x$ for $m$-a.e. $x\in D$. Thus $\gamma_A(\psi)=g$, so $\gamma_A$ is well defined. Continuity of $\gamma_A$ follows from (\ref{eq5.4}) and Remark \ref{uw.sta}.
\end{dow}
\medskip

\begin{wn}
\label{wn.tr2}
If $u$ is a solution to wDP$(A,D,\psi)$ with $\psi\in H^1(D)$, then $\gamma_A(u)=\gamma_A(\psi)$ and
\begin{equation}
\label{eq.repr}
u(x)=E_x\gamma_A(\psi)(X_{\tau_D-})\quad m\mbox{-a.e.}
\end{equation}
\end{wn}
\begin{dow}
In the proof of Theorem \ref{tw5.1} we have shown that the processes $(u\mathbf{1}_D+g\mathbf{1}_{\partial_M D})(X^*)$
and $(\psi\mathbf{1}_D+g\mathbf{1}_{\partial_M D})(X^*)$ are continuous on $[0,\tau_D]$. This implies that $g=\gamma_A(u)=\gamma_A(\psi)$. By \cite[Remark 2.13]{Kl:PA2}, the  family $\{u(X_{\tau_{D_n}})\}$ is uniformly integrable under the  measure $P_x$ for $m$-a.e. $x\in D$. From this, continuity of the process $(u\mathbf{1}_D+g\mathbf{1}_{\partial_M D})(X^*)$ and (\ref{eq5.1}) we get (\ref{eq.repr}) by the Vitali convergence theorem.
\end{dow}


\begin{tw}
\label{tw5.2}
There exists an isometric embedding
\[
i_M:L^2(\partial D;\omega^A_m)\hookrightarrow L^2(\partial_M D;h^A_m)
\]
such that for $m$-a.e. $x\in D$,
\begin{equation}
\label{eq5.5}
i_M(\psi)(X_{\tau_D-})=\psi(X_{\tau_D})\quad P_x\mbox{-a.s.}
\end{equation}
\end{tw}
\begin{dow}
Let $\psi\in L^2(\partial D;\omega^A_m)$. Set $u(x)=E_x\psi(X_{\tau_D})$, $x\in D$.
By Corollary \ref{wn2.3}, $u\in C(D)\cap H^1(D)$. By the Markov property,
\[
u(X_t)=E_x(\psi(X_{\tau_D})|\FF_t),\quad t\le\tau_D,
\]
and by the martingale convergence theorem, $u(X_t)\rightarrow \psi(X_{\tau_D})$ as $t\nearrow \tau_D$. Set
\[
Y=\limsup_{t\nearrow \tau_D} u(X_t).
\]
As in the proof of Theorem \ref{tw5.1}, we show that there exists $g\in\BB(\partial_MD)$ such that (\ref{eq.tr}) is satisfied. Putting now
\begin{equation}
\label{eq.A}
i_M(\psi)=g,
\end{equation}
we get (\ref{eq5.5}). Squaring (\ref{eq5.5}), taking the expectation with respect to $P_x$ and then integrating with respect to $x$ on $D$ shows  that $i_M$ is an isometry.
\end{dow}

\begin{uw}
\label{uw5.1ab}
From Theorem \ref{tw5.2} it follows in particular that $L^2(\partial D;\omega^A_m)$ is a closed subset of $L^2(\partial_MD;h^A_m)$.
\end{uw}

By Theorem \ref{tw5.2}, we may think of $L^2(\partial
D;\omega^A_m)$ as being a subset of $L^2(\partial_MD;h^A_m)$. This
allows us to establish some standard properties  of the trace
operator $\gamma_A$.

\begin{stw}
\label{stw5.1}
Let $\psi\in H^1(D)$. Then $\psi\in H^1_0(D)$ if and only if $\gamma_A(\psi)=0$.
\end{stw}
\begin{dow}
Assume that $\psi\in H^1_0(D)$. Then, by Lemma \ref{lm2.2},
$\psi(X_t)\rightarrow 0$ $P_x$-a.s. as $t\nearrow \tau_D$ for
$m$-a.e. $x\in D$. Therefore, by (\ref{eq.tr1}), $\gamma_A(\psi)=0$. Conversely, suppose that
$\gamma_A(\psi)=0$. Let $u$ be a  solution to wDP$(A,D,\psi)$.
Then $u-\psi\in H^1_0(D)$. By Corollary \ref{wn.tr2},
$u\equiv 0$. Thus $\psi\in H^1_0(D)$.
\end{dow}
\medskip

\begin{stw}
\label{stw5.2}
Let $\psi\in H^1(D)\cap C(\overline D)$. Then $\gamma_A(\psi)=\psi_{|\partial D}$.
\end{stw}
\begin{dow}
Let $u(x)=E_x\psi(X_{\tau_D})$. By  (\ref{eq.A}), $i_M(\psi_{|\partial D})=\gamma_A(u)$.
Therefore, it is enough to show that $\gamma_A(u)=\gamma_A(\psi)$. But by Theorem \ref{tw2.1}, $u-\psi\in H^1_0(D)$, which when
combined with Proposition \ref{stw5.1} gives the result.
\end{dow}
\medskip

Let $H^1_c(D)$ denote the set of all $u\in H^1(D)$ for  which
there exists $\psi\in \BB(\partial D)$ such that
\[
[0,\tau_D]\ni t\rightarrow (u\mathbf{1}_D
+\psi\mathbf{1}_{\partial D})(X_t)
\mbox{ is continuous }P_x\mbox{-a.s. for $m$-a.e. }x\in D.
\]

\begin{stw}
\label{stw5.3}
We have
\[
\gamma^{-1}_A(L^2(\partial D;\omega^A_m))=H^1_c(D).
\]
\end{stw}
\begin{dow}
Assume that  $\psi\in \gamma_A^{-1}(L^2(\partial D;\omega^A_m))$.
By the definition of the trace operator, the process
\[
[0,\tau_D]\ni t\rightarrow \mathbf{1}_{[0,\tau_D)}(t)\psi(X_t)+\mathbf{1}_{\{t=\tau_D\}}\gamma_A(\psi)(X_{\tau_D-})
\]
is continuous under $P_x$-a.s. for $m$-a.e. $x\in D$. But $\gamma_A(\psi)\in L^2(\partial D;\omega^A_m)$, so there exists
$\tilde{\psi}\in L^2(\partial D;\omega^A_m)$ such that $i_M(\tilde\psi)=\gamma_A(\psi)$. By (\ref{eq5.5}) we get that $\psi\in H^1_c(D)$. Now assume that $\psi\in H^1_c(D)$. Let $g\in \BB(\partial D)$ be such that
\[
[0,\tau_D]\ni t\rightarrow (\psi\mathbf{1}_D+g\mathbf{1}_{\partial D})(X_t)\mbox{ is continuous }P_x\mbox{-a.s. for a.e. }x\in D.
\]
Let $u$ be a solution to wDP$(A,D,\psi)$. By Corollary \ref{wn.tr2}, for $m$-a.e. $x\in D$ we have
\[
\gamma_A(\psi)(X_{\tau_D-})=\lim_{t\nearrow\tau_D}u(X_t),\quad P_x\mbox{-a.s.}
\]
Since $(u-\psi)\in H^1_0(D)$ by Lemma \ref{lm2.2}, $\lim_{t\nearrow\tau_D}u(X_t)=\lim_{t\nearrow\tau_D}\psi(X_t)$ $P_x$-a.s. for $m$-a.e. $x\in D$.
But $\lim_{t\nearrow\tau_D}\psi(X_t)=g(X_{\tau_D})$. Hence
\[
\gamma_A(\psi)(X_{\tau_D-})=g(X_{\tau_D})\quad P_x\mbox{-a.s.}
\]
for $m$-a.e. $x\in D$, which implies that $i_M(g)=\gamma_A(\psi)$.
\end{dow}

\begin{wn}
\label{wn5.1}
$H^1_c(D)$ is a closed subspace of $H^1(D)$.
\end{wn}

\begin{tw}
\label{tw5.3}
For every bounded open set $D\subset\BR^d$,
$
L^2(\partial D;\omega^A_m)= L^2(\partial_M D;h^A_m)$ if and only if $H^1_c(D)=H^1(D)$.
\end{tw}
\begin{dow}
Follows from Proposition \ref{stw5.3}.
\end{dow}

\begin{uw}
\label{uw5.1}
A rephrasing of Theorem \ref{tw5.3} is that for a bounded open $D\subset\BR^d$   there exists the trace operator from $H^1(D)$ to $L^2(\partial D;\omega^A_m)$ if and only if  $H^1_c(D)=H^1(D)$. Observe that
\[
\mbox{cl}(H^1(D)\cap C(\overline D))\subset H^1_c(D)
\]
since $H^1_c(D)$ is a closed subspace of $H^1(D)$. Hence, in particular, if $\mbox{cl}(H^1(D)\cap C(\overline D))=H^1(D)$, then there exists
the trace operator from $H^1(D)$ to $L^2(\partial D;\omega^A_m)$.
\end{uw}

\section{Weak Dirichlet problem vs Dirichlet problem on arbitrary domain} \label{sec6}

\begin{lm}
\label{lm6.1}
Let $\psi\in\BB_b(\partial_M D)$ and let $u(x)=E_x\psi(X_{\tau_D-})$, $x\in D$. Then $u\in C(D)$.
\end{lm}
\begin{dow}
Of course, $u\in\BB_b(D)$. Let $V\subset\subset D$.
Then, by the strong Markov property,
\[
u(X_{\tau_V})=E_{X_{\tau_V}}\psi(X_{\tau_D-})=E_x(\psi(X_{\tau_D-})|\FF_{\tau_V}).
\]
Taking the  expectation with respect to $P_x$, we  get
\begin{equation}
\label{eq6.1}
u(x)=E_xu(X_{\tau_V}),\quad x\in D.
\end{equation}
Hence,  by Corollary \ref{wn2.3}, $u\in C(V)$. Since $V\subset\subset U$ was arbitrary,   $u\in C(D)$.
\end{dow}

\begin{wn}
\label{wn6.1}
Let $D$ be a bounded domain. Then for every $B(x_0,r)\subset D$ there exists $c>0$ such that for all  $x,y\in B(x_0,r)$,
\[
c^{-1} h^A_x(dz)\le h^A_y(dz)\le c h^A_x(dz).
\]
\begin{dow}
Let $\psi\in \BB_b(\partial_MD)$ and let  $u(x)=E_x\psi(X_{\tau_D-})$, $ x\in D$. It is clear that $u$ is bounded.
Let $B(x_0,r)\subset\subset V\subset\subset D$. By (\ref{eq6.1}) and Proposition \ref{stw2.1},
$u\in H^1_{loc}(V)$ and $\EE(u,v)=0,\, v\in C^1_c(V)$. Hence, by  Harnack's inequality, there is $c>0$ such that
\[
c^{-1}u(x)\le u(y)\le cu(x),\quad x,y\in B(x_0,r).
\]
Since $c$ is independent of $\psi$, the desired result follows.

\end{dow}
\end{wn}

\begin{wn}
\label{wn6.2}
Let $D$ be a bounded open domain.
If $\psi\in L^p(\partial_MD; h^A_x)$ for some $p>0$ and $x\in D$, then $\psi\in L^p(\partial_MD; h^A_y)$ for every $y\in D$.
\end{wn}

\begin{wn}
\label{wn6.3}
Let $\psi\in L^1(\partial_MD;h^A_m)$ and let $u(x)=E_x\psi(X_{\tau_D-})$, $x\in D$. Then $u\in C(D)$.
\end{wn}
\begin{dow}
By Corollary \ref{wn6.2}, $\psi\in L^1(\partial_MD;h^A_x),\, x\in D$. Set $\psi_n=(\psi\wedge n)\vee(-n)$, $u_n(x)=E_x\psi_n(X_{\tau_D-})$, $x\in D$,
and choose $x_0\in D,\, r>0$ so that $B(x_0,r)\subset D$. By Corollary \ref{wn6.1},
\[
|u_n(x)-u(x)|\le c\int_{\partial_MD}|\psi_n(y)-\psi(y)|\,h^A_{x_0}(dy),\quad x\in B(x_0,r).
\]
Hence $u_n\rightarrow u$ uniformly on compact subsets of $D$. Since $u_n\in C(D)$ by Lemma \ref{lm6.1},  $u\in C(D)$.
\end{dow}
\medskip

In what follows, we set
$H^{1/2}(\partial_MD):=\gamma_A(H^1(D))$.

\begin{df}
Let $\psi\in H^{1/2}(\partial_MD)$. We say that  $u\in H^1(D)$ is a weak solution to the Dirichlet problem
\[
-Au=0\quad\mbox{in}\quad D,\qquad u_{|\partial_MD}=\psi\quad\mbox{on}\quad \partial_{M}D,
\]
(DP$(A,\partial_MD,\psi)$ for short) if
\[
\EE(u,v)=0,\quad v\in H^1_0(D),\qquad \gamma_A(u)=\psi.
\]
\end{df}

\begin{tw}
\label{tw6.1}
Let $\psi\in H^1(D)$. If $u\in H^1(D)\cap C(D)$ is a solution to \mbox{\rm wDP}$(A,D,\psi)$, then $u$ is a weak solution to \mbox{\rm DP}$(A,\partial_MD,\gamma_A(\psi))$. Morover,
\begin{equation}
\label{eq6.2}
u(x)=E_x\gamma_A(\psi)(X_{\tau_D-}),\quad x\in D.
\end{equation}
\end{tw}
\begin{dow}
By Corollary \ref{wn6.3}, it is sufficient  to show that (\ref{eq6.2}) holds for $m$-a.e. $x\in D$.
But this follows from Corollary \ref{wn.tr2}.
Furthermore, by the definition of a
solution of wDP and  Proposition \ref{stw5.1} we get the first
assertion of the theorem.
\end{dow}

\begin{stw}
\label{stw6.1}
Let $\psi\in L^2(\partial_M D;h^A_m)$ and let $u$ be defined  by
\begin{equation}
\label{eq6.wzor}
u(x)=E_x\psi(X_{\tau_D-}),\quad x\in D.
\end{equation}
Then,
\begin{enumerate}
\item [\rm(i)] $u\in \breve H^1_\delta(D)\cap C(D)$.
\item [\rm(ii)] $\EE(u,v)=0$ for every $v\in C_c^\infty(D)$.
\item [\rm(iii)] $\int_{\partial D_n}u(y)\,\omega^A_{x,D_n}(dy)\rightarrow \int_{\partial_M D}\psi(y)\,h^A_{x,D}(dy)$ for evey $x$ and for every increasing sequence $\{D_n\}$ of bounded open subsets of $\BR^d$ such that $\{D_n\}\subset\subset D$ and $\bigcup_{n\ge 1} D_n=D$.

\item [\rm(iv)] $|u(x)|+\|\sqrt{g_D}(x,\cdot)\nabla u\|_{L^2(D;m)}\le (1\wedge \lambda)^{-1}\|\psi\|_{L^2(\partial_M D;h^A_x)}$ for every $x\in D$.
\item [\rm(v)] $\|u\|_{\breve H^1_\delta}\le (1\wedge \lambda)^{-1}\|\psi\|_{L^2(\partial_M D;h^A_m)}$.

\item [\rm(vi)] $u(X_t)=\psi(X_{\tau_{D-}})-\int_t^{\tau_{D}}\bar\sigma\nabla u(X_r)\,dB_r$, $t\le\tau_{D}$ $P_x\mbox{-a.s.}$ for every $ x\in D$.

\item [\rm(vii)] $E_x\sup_{t<\tau_D}|u(X_t)|^2\le 4\|\psi\|^2_{L^2(\partial_M D;h^A_x)}$ for every $x\in D $.
\end{enumerate}
\end{stw}
\begin{dow}
By Corollary \ref{wn6.3}, $u\in C(D)$. By the strong Markov property (see the reasoning in the proof of Lemma \ref{lm6.1}), $u(x)=E_x u(X_{\tau_{D_n}})$ for $x\in D_n$, so  by Theorem \ref{tw2.1},
$u$ is also a solution to DP$(A,D_n,u_{|\partial D_n})$. Therefore,  by Proposition \ref{stw2.1}, $u\in H^1_{loc}$ and  assertion (ii)  holds true. Moreover, by Proposition \ref{stw2.1.1}, (\ref{eq2.13}) is satisfied.
By the Markov property,
\begin{equation}
\label{eq6.4}
u(X_t)=E_x(\psi(X_{\tau_D-})|\FF_t),\quad t\le\tau_{D},\quad P_x\mbox{-a.s.}
\end{equation}
for every $x\in D$. Applying Doob's $L^2$-inequality, we obtain
\begin{equation}
\label{eq6.5}
E_x\sup_{t\le\tau_D} |u(X_t)|^2\le 4E_x|\psi(X_{\tau_D-})|^2,\quad x\in D.
\end{equation}
This gives (vii). By (\ref{eq2.13}),
\[
[u(X)]_{\tau_{D_n}}=\int_0^{\tau_{D_n}}|\bar\sigma\nabla u|^2(X_r)\,dr\quad P_x\mbox{-a.s.}
\]
Using (vii) and the Burkholder-Davis-Gundy inequality, we obtain
\begin{equation}
\label{eq6.6abc}
E_x\int_0^{\tau_{D_n}}|\nabla u(X_r)|^2\le cE_x|\psi(X_{\tau_D-})|^2,\quad x\in D_n,\quad n\ge 1.
\end{equation}
Letting $n\rightarrow \infty$ in (\ref{eq6.6abc}) yields
\begin{equation}
\label{eq6.6}
E_x\int_0^{\tau_{D}}|\nabla u(X_r)|^2\le cE_x|\psi(X_{\tau_D-})|^2,\quad x\in D.
\end{equation}
By (\ref{eq6.4}) and the martingale convergence theorem, for every $x\in D$,
\begin{equation}
\label{eq6.7}
u(X_t)\rightarrow \psi(X_{\tau_D-}) \quad P_x\mbox{-a.s.}
\end{equation}
as $t\nearrow \tau_D$. Letting $n\rightarrow\infty$ in (\ref{eq2.13}) and using (\ref{eq6.6}) and (\ref{eq6.7}) we get  (vi).
By  (\ref{eq6.5}) and (\ref{eq6.7}),
\[
E_xu(X_{\tau_{D_n}})\rightarrow E_x\psi(X_{\tau_D}),
\]
which yields (iii). By (vi) and It\^o's formula,
\[
|u(x)|^2+E_x\int_0^{\tau_D}|\bar\sigma\nabla u(X_r)|^2\,dr=E_x|\psi(X_{\tau_D-})|^2,\quad x\in D,
\]
which implies (iv). Assertion (v) is a consequence of (iv).
\end{dow}

\begin{df}
Let $\psi\in L^2(\partial_MD; h^A_m)$.
We say that $u:D\rightarrow\BR$ is a soft solution  to DP$(A,\partial_M D,\psi)$ if (i)--(iii) of Proposition \ref{stw6.1} are satisfied.
\end{df}

\begin{stw}
\label{stw6.2}
For every $\psi\in L^2(\partial_M D;h^A_m)$ there exists a unique  soft solution to \mbox{\rm DP}$(A,\partial_M D,\psi)$.
\end{stw}
\begin{proof}
The proof is analogous to the proof of Proposition \ref{stw2.1}.
\end{proof}

\section{Extension of the trace operator} \label{sec7}

Denote by $\TT$ the subset of $\BB(D)$ consisting of all  functions $\psi$ for which there exists $g\in\BB(\partial_MD)$
such that the process $[0,\tau_D]\ni t\mapsto (\mathbf{1}_D\psi+\mathbf{1}_{\partial_MD}g)(X^*_t)$ is continuous at $\tau_D$ under the  measure $P_x$
for $m$-a.e. $x\in D$. Of course, such a function is unique up to the measure $h^A_m$ because
\begin{equation}
\label{eq.tr4}
g(X_{\tau_D-})=\lim_{t\nearrow\tau_D} \psi(X_t)\quad P_x\mbox{-a.s.}
\end{equation}
for $m$-a.e. $x\in D$. We equip $\TT$ with the metric
$\|u\|_{q.u.}$ defined by
\[
\|u\|_{q.u.}=E_m\sup_{t<\tau_D}(|u(X_t)|\wedge 1),
\]
where $P_m(dy):=\int_DP_x(dy)\,m(dx)$.
By \cite{LeJan},  if $u_n\rightarrow u$ quasi-uniformly, i.e.
\[
\lim_{N\rightarrow \infty}\mbox{Cap}_A\Big(\bigcup_{n\ge N} \{|u_n-u|>\varepsilon\}\Big)=0,\quad \varepsilon>0.
\]
then $\|u-u_n\|_{q.u.}\rightarrow 0$ . Conversely, if $\|u-u_n\|_{q.u.}\rightarrow 0$ then there exists subsequence
$(n_k)\subset (n)$ such that $u_n\rightarrow u$ quasi-uniformly.
By  $L^0(\partial_MD; h^A_m)$ we denote the space $\BB(\partial_MD)$ equipped with the metric of convergence in measure
$h^A_m$.

Consider the trace operator
\[
\gamma_A:\TT\rightarrow \BB(\partial_MD),\quad \gamma_A(\psi):=g,\quad \psi\in\TT.
\]

\begin{stw}
\label{stw7.1}
The operator
$
\gamma_A:(\TT,\|\cdot\|_{q.u.})\rightarrow L^0(\partial_MD;h^A_m)
$
is continuous.
\end{stw}
\begin{dow}
Let $u\in \TT$. Suppose that  $\|u_n\|_{q.u.}\rightarrow 0$. For all $n\ge1$ we have
\begin{align*}
\int_{\partial_MD}(|\gamma_A(u_n)|\wedge 1)(y)\,h^A_m(dy)&=E_m(|\gamma_A(u_n)|\wedge 1)(X_{\tau_D-})\le \|u_n\|_{q.u.}
\end{align*}
from which we deduce that $\gamma_A(u_n)\rightarrow0$ in measure $h^A_m$.
\end{dow}
\medskip

For $p\ge1$, we set
\[
\TT_0=\gamma_A^{-1}(0),\qquad \TT^p=\gamma_A^{-1}(L^p(\partial_MD;h^A_m)).
\]
Of course, $\TT_0\subset\TT^p$ for every $p\ge 1$.

\begin{uw}
By Theorem \ref{tw5.1}, $H^1(D)\subset \TT^2$.
\end{uw}

We denote by $\mathfrak S^p$, $p\ge1$,  the set of all  quasi-continuous $u\in\BB(D)$ having a finite norm
\[
\|u\|_{\mathfrak S^p}=(E_m\sup_{t<\tau_D}|u(X_t)|^p)^{1/p},
\]
and by $\mathfrak D^p$ we denote the set of all quasi-continuous $u\in\BB(D)$ for which the family $\{|u|^p(X_{\tau_V}):V\subset\subset U\}$ is uniformly integrable under the measure $P_x$ for $m$-a.e. $x\in D$. We equip  $\mathfrak D^p$ with the norm
\[
\|u\|_{\mathfrak{D}^p}=\Big(\int_D\sup_{V\subset\subset D}E_x|u(X_{\tau_V})|^p\, m(dx)\Big)^{1/p}.
\]
Let us stress that the  norms $ \|\cdot\|_{\mathfrak
D^p}, \|\cdot\|_{\mathfrak S^p}$ depend on the operator $A$.
It is an elementary check that  $ (\mathfrak S^p,\, \|\cdot\|_{\mathfrak S^p})$ with $p\ge 1$  are Banach spaces. It is clear that $\mathfrak S^p\subset \mathfrak D^p$ for every $p\ge1$.

\begin{stw}
$(\mathfrak D^p, \|\cdot\|_{\mathfrak
D^p})$ is a Banach space for every $p\ge 1$.
\end{stw}
\begin{dow}
Let $\{u_n\}\subset \mathfrak D^p$ be a Cauchy sequence in the norm $\|\cdot\|_{\mathfrak D^p}$.
By \cite{LeJan} (see comments following Lemma 1 in \cite{LeJan}), for every $n\ge1$ there exists $\psi_n\in C_c(D)$ such that
\begin{equation}
\label{eq.upc}
\||u_n-\psi_n|\wedge 1\|^p_{\mathfrak D^p}\le E_m\sup_{t<\tau_D}|u_n(X_t)-\psi_n(X_t)|^p\wedge 1\le 2^{-2pn}
\end{equation}
(the first inequlity above is obvious).
Set $V_{n,k}=\{|\psi_n-\psi_k|\wedge 1<\varepsilon\}\cap D$, and let $\{D_l,\,l\ge 1\}$ be an increasing sequence of open subsets of $D$ such that $D_l\subset\subset D$ and $\bigcup_{l\ge 1} D_l=D$. Observe that $\tau_{V_{n,k}\cap D_l}\rightarrow \tau_{V_{n,k}}$ as $l\rightarrow \infty$. For $\varepsilon\le 1$ and $n\le k$ we have
\begin{align*}
(P_m(\sup_{t<\tau_D}|\psi_n-\psi_k|(X_t)>\varepsilon))^{1/p}
&=(P_m(\sup_{t<\tau_D}|\psi_n-\psi_k|\wedge 1(X_t)>\varepsilon))^{1/p}\\&
\le (P_m(|\psi_n-\psi_k|\wedge 1(X_{\tau_{V_{n,k}}})\ge \varepsilon))^{1/p}\\
&\le \varepsilon^{-1}(E_m|\psi_n-\psi_k|^p\wedge 1(X_{\tau_{V_{n,k}}}))^{1/p}\\
&=\lim_{l\rightarrow\infty}\varepsilon^{-1}[E_m|\psi_n-\psi_k|^p\wedge 1(X_{\tau_{V_{n,k}\cap D_l}})]^{1/p}\\
&\le\varepsilon^{-1}\||\psi_n-\psi_k|\wedge 1\|_{\mathfrak D^p}\\
&\le\varepsilon^{-1}(2^{-2n+1}+\|u_n-u_k\|_{\mathfrak D^p}).
\end{align*}
By the above inequality and (\ref{eq.upc}),
\begin{equation}
\label{eq.upc1}
(P_m(\sup_{t<\tau_D}|u_n-u_k|(X_t)>\varepsilon))^{1/p}\le \varepsilon^{-1}(2^{-2n+2}+\|u_n-u_k\|_{\mathfrak D^p}).
\end{equation}
Let $\{n_k\}$ be a sequence such that $\|u_{n_k}-u_{n_{k+1}}\|_{\mathfrak D^p}\le 2^{-2k}$. Then by (\ref{eq.upc1}),
\[
P_m(\sup_{t<\tau_D}|u_{n_k}-u_{n_{k+1}}|(X_t)>2^{-k})\le 2^{-p(k-3)}.
\]
From this and the Borel-Cantelli lemma we deduce that
\begin{equation}
\label{eq7.6}
P_m(\limsup_{k\ge 1} \Lambda_k)=0,
\end{equation}
where $\Lambda_k=\{\omega\in\Omega;\, \sup_{t<\tau_D}|u_{n_k}-u_{n_{k+1}}|(X_t)>2^{-k}\}$.
Set $u=\liminf_{k\rightarrow} u_{n_k}$. From (\ref{eq7.6})
it follows that for every $\varepsilon>0$,
\begin{equation}
\label{eq.upc3}
\lim_{k\rightarrow \infty}P_m(\sup_{t<\tau_D}|u_{n_k}(X_t)-u(X_t)|>\varepsilon)=0.
\end{equation}
Hence, in particular, $u$ is quasi-continuous (see \cite[Theorem 1]{LeJan}).
Since $\{u_n\}$ is a Cauchy sequence in $\mathfrak D^p$, there exists $m_0\in\mathbb N$ such that
for all $k,l\ge m_0$ and  $R\ge0$,
\[
\int_D\sup_{V\subset\subset D}E_x(|u_{n_k}(X_{\tau_V})-u_{n_l}(X_{\tau_V})|^p\wedge R )\,m(dx)\le\varepsilon^p.
\]
Letting $l\rightarrow \infty$ in the above inequality and using
(\ref{eq.upc3}) and the Lebesgue dominated convergence theorem, we get
\[
\int_D\sup_{V\subset\subset D}E_x(|u_{n_k}(X_{\tau_V})-u(X_{\tau_V})|^p\wedge R )\,m(dx)\le\varepsilon^p.
\]
Now, letting $R\rightarrow \infty$, we get $\|u_{n_k}-u\|_{\mathfrak D^p}\le\varepsilon$ for $k\ge m_0$, from which the desired result follows.
\end{dow}
\medskip

We set
\[
H_{\mathfrak D^p}=H(D)\cap \mathfrak D^p,\qquad H_{\mathfrak S^p}
=H(D)\cap \mathfrak S^p,
\]
where $H(D)$ is defined by (\ref{eq3.6}).

\begin{tw}
\label{tw7.1}
$H_{\mathfrak D^p}\subset\TT^p$ for every $p\ge 1$. Moreover,
\begin{enumerate}
\item[\rm(i)]$\gamma_A:H_{\mathfrak D^p}\rightarrow L^p(\partial_MD; h^A_m)$ is an isometric isomorphism,
\item[\rm(ii)]for every $p>1$, $\gamma_A:H_{\mathfrak S^p}\rightarrow L^p(\partial_MD;h^A_m)$ is a homeomorphism,

\item[\rm(iii)]if $u\in H_{\mathfrak D^1}$, then $ E_x(\int_0^{\tau_D}|\nabla u |^2(X_r)\,dr)^{q/2}\le \lambda^{-q}c_q\|\gamma_A(u)\|^q_{L^1(\partial_MD;h^A_x)}<\infty$ for all $x\in D$ and $q\in(0,1)$, and moreover, there exists a Wiener process $B$ such that for every $x\in D$,
\begin{equation}
\label{eq7.12}
u(X_t)=\gamma_A(u)(X_{\tau_D-})-\int_t^{\tau_D}\bar\sigma\nabla u(X_r)\,dB_r,\quad t\le\tau_D\quad P_x\mbox{-a.s.},
\end{equation}
\item[\rm(iv)] if $p>1$ and $u\in H_{\mathfrak S^p}$, then $E_x(\int_0^{\tau_D}|\nabla u|^2(X_r)\,dr)^{p/2}\le \lambda^{-p}c_p\|\gamma_A(u)\|^p_{L^p(\partial_MD;h^A_x)}<\infty$ for every $x\in D$,
\item[\rm(v)]for every $p>1$,  $H_{\mathfrak D^p}=H_{\mathfrak S^p}$ and $\|u\|_{H_{\mathfrak D^p}}\le  \|u\|_{H_{\mathfrak S^p}}\le  \frac{p-1}{p}\|u\|_{H_{\mathfrak D^p}}$ for any $u\in H_{\mathfrak D^p}$.
\end{enumerate}
\end{tw}
\begin{dow}
(i) Let $u\in H_{\mathfrak D^p}$ and $\{D_n\}$ be an increasing sequence of open sets such that $D_n\subset\subset D$ and $\bigcup_{n\ge 1} D_n=D$. Since $u\in H(D)$, $u(x)=E_xu(X_{\tau_{D_n}})$, $x\in D_n$. Hence,  by Theorem \ref{tw2.1} and Proposition \ref{stw2.1.1},
\begin{equation}
\label{eq7.0} u(X_t)=u(x)+\int_0^{t}\bar\sigma\nabla
u(X_r)\,dB_r,\quad t<\tau_D,\quad P_x\mbox{-a.s.}
\end{equation}
for every $x\in D$.  Furthermore,
since $u\in\mathfrak D^p$,  $u(X)$ is a uniformly integrable martingale  on $[0,\tau_D)$ under the measure $P_x$ for $m$-a.e. $x\in D$. Let
\[
Y=\limsup_{t\nearrow \tau_D}u(X_t).
\]
Then, by the martingale convergence theorem, $Y=\lim_{t\nearrow \tau_D} u(X_t)$ $P_m$-a.s. Since $|u|^p(X)$ is uniformly integrable under measure $P_m$  ($u\in H_{\mathfrak D^p}$) we have that $Y=\lim_{t\nearrow \tau_D} u(X_t)$ in $L^p(\Omega, P_m)$. It is clear that $Y\in\mathcal I$, so there exists $g\in\BB(\partial_MD)$
such that $g(X_{\tau_D-})=Y$ $P_m$-a.s. Since $Y\in L^p(\Omega, P_m) $, $g\in L^p(\partial_MD;h^A_m)$.
It is clear that $u\in \TT^p$ and $\gamma_A(u)=g$. Hence
\begin{equation}
\label{eq7.2}
u(x)=E_x\gamma_A(u)(X_{\tau_D-})
\end{equation}
for $m$-a.e. $x\in D$. Thus $\gamma_A$ is an injection. Let $\psi\in L^p(\partial_MD;h^A_m)$, and let
\begin{equation}
\label{eq7.3}
u(x)=E_x\psi(X_{\tau_D-}),\quad x\in D.
\end{equation}
By the Markov property,
\begin{equation}
\label{eq7.4}
u(X_t)=E_x(\psi(X_{\tau_D-})|\FF_t),\quad t\le \tau_D.
\end{equation}
From this and the martingale convergence theorem we conclude that  $\gamma_A(u)=\psi$. Furthermore, from (\ref{eq7.2}) and the strong Markov property  it follows that for every $V\subset\subset U$,
\[
u(X_{\tau_V})=E_x(\psi(X_{\tau_D-})|\FF_{\tau_V}),\quad x\in D.
\]
(see the reasoning in the proof of Lemma \ref{lm6.1}). Taking the  expectation shows that $u\in H(D)$.  Moreover, by (\ref{eq7.4}), $u(X)$ is a martingale on $[0,\tau_D]$, so $|u(X)|^p$
is a submartingale on $[0,\tau_D]$ under the measure $P_x$ for every $x\in D$. Hence
\begin{align}
\label{eq7.5}
\|u\|^p_{\mathfrak D^p}&=\int_D\sup_{V\subset\subset D}E_x|u(X_{\tau_V})|^p\,m(dx)\nonumber\\
&=\int_D|\psi(X_{\tau_D-})|^p\,m(dx)=\|\psi\|^p_{L^p(\partial_MD;h^A_m)}\,,
\end{align}
which completes the proof of (i).

(ii) Suppose that $u\in H_{\mathfrak S^p}$ for some $p>1$. Since
$H_{\mathfrak S^p}\subset H_{\mathfrak D^p}$, from the  proof of part (i) we know that $H_{\mathfrak S^p}\subset \TT^p$ and  (\ref{eq7.2}) is satisfied. Thus $\gamma_A$ is an injection.
Assume that  $\psi\in L^p(\partial_MD;h^A_m)$ and define  $u$ by
(\ref{eq7.3}). By the proof of part (i), $\gamma_A(u)=\psi$ and $u\in H_{\mathfrak D^p}$. By (\ref{eq7.4}) and Doob's $L^p$-inequality,
\begin{equation}
\label{eq7.homeo}
E_m|\psi(X_{\tau_D-})|^p\le E_m\sup_{t<\tau_D}|u(X_t)|^p\le\frac{p}{p-1}E_m|\psi(X_{\tau_D-})|^p,
\end{equation}
which completes the proof of (ii).

(iii) By Remarks \ref{wn6.2} and \ref{wn6.3}, equation (\ref{eq7.2}) holds true for every $x\in D$. Hence, by the Markov property, for every $x\in D$ we have
\[
u(X_t)=E_x(\gamma_A(u)(X_{\tau_D-})|\FF_t),\quad t\le \tau_D,\quad P_x\mbox{-a.s.}
\]
It follows that in fact $u(X)$ is a uniformly integrable martingale on $[0,\tau_D]$ under $P_x$ for every $x\in D$. Furthermore, by the martingale convergence theorem, $\gamma_A(u)(X_{\tau_D-})=\lim_{t\nearrow\tau_D }u(X_t)$ $P_x$-a.s. and in $L^1(\Omega,P_x)$ for every $x\in D$. By (\ref{eq7.0}),
\cite[Lemma 6.1]{BDHPS} and the Burkholder-Davis-Gundy inequality,
\[
E_x(\int_0^{\tau_{D_n}}|\bar\sigma\nabla u(X_t)|^2\,dt)^{q/2}\le c_qE_x\sup_{t\le\tau_{D_n}}|u(X_t)|^q\le\frac{c_q}{1-q}(E_x|u(X_{\tau_{D_n}})|)^q
\]
for all $x\in D$ and $q\in(0,1)$. Using now Fatou's lemma and Corollary \ref{wn6.2}, we get the inequality appearing in (iii). Furthermore, letting $t\nearrow \tau_D$ in (\ref{eq7.0}) we  obtain (\ref{eq7.12}).  Part (iv) follows immediately from (\ref{eq7.12}), Corollary \ref{wn6.2}  and the Burkholder-Davis-Gundy inequality. Part (v) follows from (i), (ii)
and (\ref{eq7.homeo}).
\end{dow}

\begin{wn}
\label{wn7.bound}
Assume that $p\ge 1$ and
\begin{equation}
\label{eq.tr7}
u(x)=E_x\psi(X_{\tau_D-}),\quad x\in D.
\end{equation}
Then $u\in H_{\mathfrak D^p}$
if and only if $\psi\in L^p(\partial_MD;h^A_m)$, and if $u\in  H_{\mathfrak D^p}$,
then $\psi =\gamma_A(u)$. Moreover, any $u\in H_{\mathfrak D^1}$
is of the form \mbox{\rm(\ref{eq.tr7})} with $\psi=\gamma_A(u)$.
\end{wn}
\begin{dow}
If $u\in H_{\mathfrak D^1}$, then taking the expectation with respect to $P_x$ in (\ref{eq7.12}) with $t=0$ (this is possible since $u\in H_{\mathfrak D^1}$) we get (\ref{eq.tr7}) with $\psi=\gamma_A(u)$. Let $u$ be of the form (\ref{eq.tr7}).  Assume that $u\in H_{\mathfrak D^p}$. Then, by Theorem \ref{tw7.1}(iii),
$u(X_{\tau_D})=\gamma_A(u)(X_{\tau_D-})$ $P_x$-a.s. for  $x\in D$. On the other hand, by (\ref{eq.tr7}) and the Markov property, for every $x\in D$,
\[
u(X_t)=E_x(\psi(X_{\tau_D-})|\FF_t),\quad t\le \tau_D,\quad P_x\mbox{-a.s.}
\]
Hence $\psi(X_{\tau_D-})=\gamma_A(u)(X_{\tau_D-})$ $P_x$-a.s. for  $x\in D$, which implies that $\gamma_A(u)=\psi$. By Theorem \ref{tw7.1}, $\gamma_A(u)\in L^p(\partial_MD;h^A_m)$. Now suppose that $\psi\in L^p(\partial_MD;h^A_m)$. By Theorem \ref{tw7.1}(i) there exists $\bar u\in H_{\mathfrak D^p}$ such that
$\gamma_A(\bar u)=\psi$. Moreover, by Theorem \ref{tw7.1}(iii), $\bar u(x)=E_x\gamma_A(\bar u)(X_{\tau_D-})$ for $x\in D$. Thus $u=\bar u$.
\end{dow}

\begin{uw}
From the definition of the trace operator it follows that for every $p\ge 1$
\[
\|\gamma_A(u)\|_{L^p(\partial_MD;h^A_m)}\le \|u\|_{\mathfrak D^p},\quad u\in\mathfrak D^p\cap\TT.
\]
\end{uw}

\begin{df}
We say that a function  $u$ on $D$ is of potential type if $u\in\TT_0$.
\end{df}

\begin{stw}
\label{stw7.2}
Each potential is a function of potential type.
\end{stw}
\begin{dow}
Let $u$ be a potential. Without loss of generality, we may assume
that $u$ is bounded. It is well know   that
$u_n:=nR^D_n u\nearrow u$ quasi-uniformly as $n\rightarrow\infty$.
Hence, by Proposition \ref{stw7.1}, $\gamma_A(u_n)\rightarrow
\gamma_A(u)$ in measure $h^A_m$. Since $u_n\in H^1_0(D)$,
$\gamma_A(u_n)=0$, and consequently  $\gamma_A(u)=0$.
\end{dow}
\begin{stw}
\label{stw7.3} For every $\psi\in\TT^1$ there  exist a unique
harmonic function $h\in H_{\mathfrak D}$ and a function $p$ of
potential type such that
\[
\psi=h+p.
\]
\end{stw}
\begin{dow}
Define $h$  by $h(x)=E_x\gamma_A(u)(X_{\tau_D-}),\, x\in D$, and  set
$p:=\psi-h$. By Corollary \ref{wn7.bound} $h\in H_{\mathfrak D}$ and $\gamma_A(p)=0$. Now, suppose that  there is another
harmonic function $h'\in H_{\mathfrak D}$ and a function $p'$ of potential type such that $u=h'+p'$. Then $p-p'=h'-h$, and hence by Proposition \ref{stw7.2}
$\gamma_A(h)=\gamma_A(h')$. By this and Theorem \ref{tw7.1},
$h=h'$, and, in consequence, $p=p'$.
\end{dow}

\section{Dirichlet problem for semilinear equations
with measure data}
\label{sec8}

Let $\mu$ be a Borel measure on $D$. In this  section we are concerned with the problem of existence of solutions of the Dirichlet problem \begin{equation}
\label{eq8.0}
-Au=f(\cdot,u)+\mu\quad \mbox{in}\quad D,
\qquad u=\psi\quad\mbox{on}\quad \partial_M D.
\end{equation}
In (\ref{eq8.0}), $f:D\times\BR\rightarrow \BR$ is a  measurable
functions which is continuous and nonincreasing with respect to
$u$, and $\mu$ is a Radon signed measure on $D$. It is  known (see \cite{FST}) that  $\mu$ admits unique decomposition of the form
\[
\mu=\mu_c+\mu_d
\]
into the measure $\mu_c$, which is   singular with respect to
Cap$_A$ (the concentrated part of $\mu$) and the measure  $\mu_d$,
which is absolutely continuous  with respect to Cap$_A$ (the
diffuse part of $\mu$). In the sequel, we set $L^p_\delta(D;m):=L^p(D;\delta\cdot m)$ for $p>0$, and
we denote by $W^{1,q}_\delta(D)$ the space of functions $u\in W^{1,q}_{loc}(D)$ such that
\[
\|u\|_{W^{1,q}_\delta}:=\|u\|_{L^q_\delta(D;m)}+\|\nabla u\|_{L^q_\delta(D;m)}<\infty.
\]

Semilinear problems of the form (\ref{eq8.0})  with $\psi=0$  were
for the first time considered in the paper by Brezis and Strauss \cite{BS}
in the case  where $A=\Delta$ and  $\mu\in L^1(D;m)$ (see also
\cite{Konishi}). An important contribution to the  theory was made
in the paper \cite{BBGGPV}, in which equations of the form
(\ref{eq8.0}) with  zero boundary data but general bounded smooth
measure and operator of the form (\ref{eq1.1}) are considered. At
present, in the case where $\psi=0$, existence, uniqueness and
regularity results are known for (\ref{eq8.0}) with general
bounded smooth measure and general, possibly nonlocal, operator
$A$ corresponding to a Dirichlet form (see \cite{KR:JFA,KR:CM}).

The case $\mu_c\neq0$ is much more involved. In  1975 B\'enilan
and Brezis considered (\ref{eq8.0}) with $A=\Delta$,  $\psi=0$ and
$\mu=\delta_a$ for some $a\in D$. They showed that if $d\ge3$ and
$f(u)=u|u|^{p-1}$ with $p>\frac{d}{d-2}$,  then there is no
solution to (\ref{eq8.0}) (see \cite{BB} for interesting
historical comments on the problem). To analyze the nonexistence
phenomena behind the semilinear Dirichlet problem, Brezis, Marcus
and Ponce \cite{BMP1,BMP} introduced the concept of  good measure
for (\ref{eq8.0}), i.e. a measure for which there exists a
solution to (\ref{eq8.0}), and the concept of reduced measure for
$\mu$, i.e. the largest good measure which is less then or equal
to $\mu$. In \cite{Kl:CVPDE} the notions of good and  reduced
measure were extended to (\ref{eq8.0}) with general Dirichlet
operator and $\psi=0$.

In what follows, we concentrate on (\ref{eq8.0}) with $A$ defined
by (\ref{eq1.1}) and nonzero boundary condition $\psi$. As for
$\mu$, we will assume that it belongs to the space $\MM_\delta$ of
all signed  Borel measures on $D$ such that
$\|\mu\|_{TV,\delta}:=\int_D\delta\,d|\mu|<\infty$, where $|\mu|$
denotes the variation of $|\mu|$. Note that $\MM_{\delta}$
includes all bounded Radon measures on $D$.

Denote by  $\mathcal G_\psi(A,f)$ the set of all good  measures
for $A$ and $f$, i.e the set of all $\mu\in\MM_\delta$ for which
there exists a solution to (\ref{eq8.0}). By $\GG_0(A,f)$ we
denote the set $\GG_{\psi}(A,f)$ with $\psi=0$. We will show that,
under some assumptions on $f$,
\begin{equation}
\label{eq8.9}
\GG_0(A,f)=\GG_{\psi}(A,f).
\end{equation}
This extends in part the results of \cite{Kl:CVPDE} where considered   problem of the form (\ref{eq8.0}) with  $\psi\equiv 0$ and general Dirichlet operators, and \cite{MarcusPonce} where it was shown (for the Laplace operator) that reduced measure (the biggest measure less than $\mu$ for which there exists a solution to (\ref{eq8.0})) does not depend on the boundary conditions. Thanks to (\ref{eq8.9}) we may apply to (\ref{eq8.0}) the results of \cite{Kl:CVPDE}, where we proved some characterization of the set $\GG_0(A,f)$.

To formulate our assumptions on $f $ and prove (\ref{eq8.9}), we
will need the notions of quasi-integrable and quasi-bounded
function, which we define below.

We say that $u\in \BB(D)$ is quasi-integrable ($u\in qL^1(D;m)$ in
abbreviation)
if for q.e. $x\in D$,
\[
P_x\Big(\int_0^{\tau_D}|f(X_r)|\,dr<\infty\Big)=1.
\]
Note that, by \cite[Theorem 4.2.5]{FOT} if for every $\varepsilon>0$ there exists a Borel
set $B_\varepsilon\subset D$ such that Cap$_A(D\setminus
B_\varepsilon)<\varepsilon$ and $u\in L^1(B_ \varepsilon;m)$ then $u\in qL^1(D;m)$.

We say that $u\in \BB(D)$ is quasi-bounded if  for q.e. $x\in D$,
\[
P_x(\sup_{t<\tau_D}|u(X_t)|<\infty)=1.
\]
By \cite[Theorem 4.2.5]{FOT} if for  every
$\varepsilon>0$ there is a Borel set $B_\varepsilon\subset D$ and
a constant $M_{\varepsilon}>0$ such that Cap$(D\setminus
B_\varepsilon)<\varepsilon$ and  $|u(x)|\le M_\varepsilon$ for
every $x\in B_\varepsilon$ then $u$ is quasi-bounded.

In the rest of this section, unless explicitly otherwise stated,
we assume that $f,\mu,\psi$ satisfy the following assumptions.
\begin{enumerate}
\item[(A1)] $\mu\in\MM_\delta$ and $\psi\in L^1(\partial_MD;h^A_m)$.

\item[(A2)]$f:D\times\BR\rightarrow \BR$ is a measurable  function such that $y\mapsto f(x,y)$ is continuous and non-increasing  for every $x\in D$, $f(\cdot,y)\in qL^1(D;m)$ for every $y\in\BR$ and $f(\cdot,0)\in L^1_\delta(D;m)$.
\end{enumerate}

It is well known that for every $\eta\in L^\infty(D;m)$,
$R^D\eta\in W^{1,\infty}_0(D)\cap C(D)$. It is also clear  that $|R^D\eta|\le \|\eta\|_\infty\delta$.

\begin{df}
We say that a  $u\in L^1(D;m)$ is a solution of the problem
\begin{equation}
\label{eq8.1}
-Au=\mu\quad\mbox{in}\quad D,\qquad u_{|\partial D}=0
\end{equation}
if
\[
\int_D u\eta\,dm=\int_DR^D\eta\,d\mu,\quad \eta\in L^\infty(D;m).
\]
\end{df}

\begin{uw}
\label{uw8.vers}
By \cite[Proposition 4.12]{Kl:CVPDE}, $u$ is a solution to (\ref{eq8.1})  if and only if $u=R^D\mu$ $ m$-a.e.
Moreover by \cite[Theorem 4.6.1]{FOT} function $R^D\mu$ is quasi-continuous and if $\mu$ is bounded then by \cite{Stampacchia}  $u\in W^{1,q}_0$ with  $q<\frac{d}{d-1}$ and
\[
\|u\|_{W^{1,q}_0}\le c_q\|\mu\|_{TV}.
\]
\end{uw}
In the sequel, we consider the operator $\gamma_A$ with domain
$H_{\mathfrak D^1}$.
\begin{df}
\label{df8.1} We say that $u\in L^1(D;m)$ is a solution  to
(\ref{eq8.0}) if $f(\cdot,u)\in L^1_\delta(D;m)$ and
$u-\gamma_A^{-1}(\psi)$ is a solution to (\ref{eq8.1}) with $\mu$
replaced by $f(\cdot,u)\cdot m+\mu$.
\end{df}

\begin{stw}
\label{stw8.1}
There exists at most one solution to \mbox{\rm(\ref{eq8.0})}.
\end{stw}
\begin{dow}
By the definition, $u$ is a solution of (\ref{eq8.0}) if and only
if $v:=u-\gamma^{-1}_A(\psi)$ is a solution of
\begin{equation}
\label{eq8.4} -Av=f_{\gamma^{-1}_A(\psi)}(\cdot,v) +\mu\quad
\mbox{in}\quad D,\qquad u=0\quad\mbox{on}\quad \partial D,
\end{equation}
where
\[
f_{\gamma^{-1}(\psi)}(x,y)=f(x,y+\gamma_A^{-1}(\psi)(x)).
\]
Therefore the desired result follows  from \cite[Corollary 4.3]{Kl:CVPDE}.
\end{dow}

\begin{uw}
\label{uw8.defd}
By Corollary \ref{wn7.bound}  $\gamma_A^{-1}(\psi)=E_\cdot\psi(X_{\tau_D-})$. So, by  Remark \ref{uw8.vers},  $u$ is  a solution
to (\ref{eq8.0}) if and only if $f(\cdot,u)\in L^1_\delta(D;m)$
and for $m$-a.e. $x\in D$,
\[
u(x)=\int_Df(y,u(y))g_D(x,y)\,m(dy)+\int_Dg_D(x,y)\,\mu(dy)
+\int_{\partial_M D}\psi(y)h^A_x(dy).
\]
If $\mu$ is smooth, then the above equation is equivalent
to
\[
u(x)=E_x\int_0^{\tau_D}f(X_r,u(X_r))\,dr+E_x\int_0^{\tau_D}\,dA^\mu_r+E_x\psi(X_{\tau_D-})
\]
(see (\ref{eq2.as}), (\ref{eq2.afr}) and  (\ref{eq5.mhm})). Therefore, under the notation of Proposition \ref{stw8.1}, $u$ is a solution to (\ref{eq8.0}) with smooth $\mu$ if and only if  $f(\cdot,u)\in L^1_\delta(D;m)$
and for $m$-a.e. $x\in D$ we have
\begin{equation}
\label{eq8.4.1}
v(x)=E_x\int_0^{\tau_D}f_{\gamma^{-1}(\psi)}(X_r,v(X_r))\,dr+E_x\int_0^{\tau_D}\,dA^\mu_r.
\end{equation}
The above formula is in  agreement with the probabilistic definition of a solution to (\ref{eq8.4}) considered in \cite{KR:JFA} if we replace ``$m$-a.e." by ``q.e." In fact, (\ref{eq8.4.1}) holds true for q.e. $x\in D$ after replacing the left-hand side of (\ref{eq8.4.1})  by its quasi-continuous $m$-version. Such a version  exists by \cite[Lemma 4.3]{KR:JFA} and is given by the right-hand side of (\ref{eq8.4.1}), which is finite for q.e. $x\in D$ by \cite[Lemma 4.2]{KR:JFA}.
\end{uw}

\begin{uw}
By \cite[Theorem 3.3]{Kl:CVPDE}, if $u$ is a solution  to (\ref{eq8.0}) and
$f(\cdot,u)\in L^1(D;m),$  $\|\mu\|_{TV}<\infty$, then
$T_k(u-\gamma^{-1}_A(\psi))\in H^1_0(D)$ for every $k\ge 0$.
\end{uw}

\begin{stw}
\label{stw8.3} If $\mu\in\MM_\delta$ is smooth, then there
exists a unique solution of \mbox{\rm(\ref{eq8.0})}.
\end{stw}
\begin{dow}
Uniqueness follows from  Proposition \ref{stw8.1}. Write $f_{n,m}=(f\wedge n)\vee(-m)$. From
(\ref{eq8.4}), (\ref{eq8.4.1}) and  \cite[Theorem
4.7]{KR:JFA} it follows that for all $n,m\ge1$ there exists a unique solution $u_{n,m}$ of (\ref{eq8.0}) with $f$ replaced by $f_{n,m}$. We may and will assume that each $u_{n,m}$ is quasi-continuous. By (\ref{eq8.4}), (\ref{eq8.4.1})  and \cite[Proposition 4.9]{KR:JFA}, $u_{n,m}\ge u_{n,m+1}$, $n,m\ge 1$, q.e.
By \cite[Theorem 4.7]{KR:JFA}, $u_{n,m}\in\mathfrak D^1$ and there exists a martingale  $M^{n,m,x}$ such that for q.e. $x\in D$,
\begin{align*}
u_{n,m}(X_t)&=\psi(X_{\tau_D-})
+\int_t^{\tau_D}f_{n,m}(X_r,u_{n,m}(X_r))\,dr\nonumber\\
&\quad+\int_t^{\tau_D}\,dA^\mu_r -\int_t^{\tau_D}\,dM^{n,m,x}_r, \quad
t\le\tau_D,\quad P_x\mbox{-a.s.}
\end{align*}
Applying Tanaka's formula, we see that for q.e. $x\in D$,
\begin{align}
\label{eq8.10.0}
|u_{n,m}(X_t)|+\int_t^{\tau_D}\,dL_r&=|\psi(X_{\tau_D-})|
+\int_t^{\tau_D}\mbox{sgn}(u_{n,m}(X_r))f_{n,m}(X_r,u_{n,m}(X_r))\,dr\nonumber\\
&\quad+\int_t^{\tau_D}\mbox{sgn}(u_{n,m}(X_r))\,dA^\mu_r\nonumber \\ &\quad-\int_t^{\tau_D}\mbox{sgn}(u_{n,m}(X_r))\,dM^{n,m,x}_r, \quad
t\le\tau_D,\quad P_x\mbox{-a.s.,}
\end{align}
where $L$ is the symmetric local time of $u_{n,m}(X)$ at zero. Taking the expectation in (\ref{eq8.10.0}) with $t=0$ and using (A2) we get
\begin{align}
\label{eq8.10.1}
&|u_{n,m}(x)|+E_x\int_0^{\tau_D}|f_{n,m}(X_r,u_{n,m}(X_r))|\,dr\nonumber \\
&\quad\le
E_x|\psi(X_{\tau_D-})|+E_x\int_0^{\tau_D}|f(X_r,0)|\,dr
+E_x\int_0^{\tau_D}\,dA^{|\mu|}_r=:v(x).
\end{align}
By the above inequality and (A2), for all $n',m',n,m\ge1$ we have
\begin{equation}
\label{eq8.10.2}
|f_{n',m'}(x,u_{n,m}(x))|\le |f(x,v(x))|+|f(x,-v(x))|,\quad x\in D.
\end{equation}
The functions $E_\cdot\int_0^{\tau_D}|f(X_r,0)|\,dr$ and $E_\cdot\int_0^{\tau_D}\,dA^{|\mu|}_r$, as potentials of a symmetric
process, are quasi-continuous (see \cite{LeJan}). Moreover, by Proposition \ref{stw7.2}, they belong to $\TT_0$. By Corollary \ref{wn7.bound},
$E_\cdot|\psi(X_{\tau_D-})|\in H_{\mathfrak D^1}$ and belongs to $\TT^1$. Therefore $v$ is quasi-continuous
and quasi-bounded. Set $V_k=\{v>k\}$. Since $v$ is quasi-bounded, the sequence $\{\tau_{V_k}\}$ is
a chain. Set
\[
\sigma^k_l=\inf\{t\ge 0,\, \int_0^t|f(X_r,k)|\,dr+\int_0^t|f(X_r,-k)|\,dr\ge l\},\quad \tau_{k,l}=\tau_{V_k}\wedge\sigma^k_l.
\]
By (A2), $\{\sigma^k_l\}$ is a chain with respect to   $l$ (with fixed $k$). Let $u_n=\inf_{m\ge 1} u_{n,m}$.
By the construction  of $\{\sigma^k_l\}$ and the Lebesgue dominated convergence theorem,
\begin{equation}
\label{eq8.10.3}
E_x\int_0^{\sigma^k_l}
|f_{n,m}(X_r,u_{n,m}(X_r))-f_n(X_r,u_n(X_r))|\,dr\rightarrow0
\end{equation}
as $m\rightarrow\infty$. By (\ref{eq8.10.0}), for q.e. $x\in D$,
\[
u_{n,m}(x)=E_xu_{n,m}(X_{\sigma^k_l})
+E_x\int_0^{\sigma^k_l}f_{n,m}(X_r,u_{n,m}(X_r))\,dr+E_x\int_0^{\sigma^k_l}\,dA^\mu_r.
\]
Letting $m\rightarrow\infty$ in the above equality and using  (\ref{eq8.10.3}) we get
\begin{equation}
\label{eq8.8888}
u_{n}(x)=E_xu_{n}(X_{\sigma^k_l})+E_x\int_0^{\sigma^k_l}f_{n}(X_r,u_{n}(X_r))\,dr
+E_x\int_0^{\sigma^k_l}\,dA^\mu_r
\end{equation}
for q.e. $x\in D$. It is clear that $E_\cdot\int_0^{\tau_D}|f(X_r,0)|\,dr$, $E_\cdot\int_0^{\tau_D}\,dA^{|\mu|}_r\in \mathfrak D^1$, so
by Corollary \ref{wn7.bound}, $v\in \mathfrak D^1$. From this and (\ref{eq8.10.1}) we conclude that the family $\{u_n(X_\tau),\,\tau\in\TT\}$ is uniformly integrable under the measure $P_x$ for $m$-a.e. $x\in D$. Therefore,  by the properties of the sequence $\{\sigma^k_l\}$,  for $m$-a.e. $x\in D$ we have
\[
\lim_{k\rightarrow \infty}\lim_{l\rightarrow \infty} E_xu_n(X_{\sigma^k_l})=E_x\psi(X_{\tau_D-}).
\]
By (\ref{eq8.10.1}), (A1), (A2) and Fatou's lemma, $f(\cdot,u_n)\in L^1_\delta(D;m)$. Hence, in particular,
$E_x\int_0^{\tau_D}|f_{n}(X_r,u_{n}(X_r))|\,dr<\infty$ for $m$-a.e. $x\in D$. Therefore letting $l\rightarrow \infty$  and then $k\rightarrow\infty$
in (\ref{eq8.8888}) we see that for  $m$-a.e. $x\in D$,
\begin{equation}
\label{eq8.12}
u_{n}(x)=E_x\psi(X_{\tau_D-})+E_x\int_0^{\tau_D}f_{n}(X_r,u_{n}(X_r))\,dr
+E_x\int_0^{\tau_D}\,dA^\mu_r,
\end{equation}
which,  by Remark \ref{uw8.defd}, shows that $u_n$ is a solution to (\ref{eq8.0}) with $f$ replaced by $f_n$.
Letting $n\rightarrow\infty$ in (\ref{eq8.12}) and using the arguments similar to those used above we show that $u$ is a solution of (\ref{eq8.0}).
\end{dow}

\begin{wn}
\label{wn8.1} For every $h\in H_{\mathfrak D}$ and  every quasi-bounded $v\in \BB(D)$  such  that $f(\cdot,v)\in L^1_\delta(D;m)$
there exists $g\in L^1_\delta(D;m)$ with the property that
$f(\cdot,v-h+R^Dg)\in L^1_\delta(D;m)$.
\end{wn}
\begin{dow}
By Theorem \ref{tw7.1}, there is $\psi\in L^1(\partial_MD;h^A_m)$
such that $h=\gamma^{-1}_A(\psi)$. By the assumptions on $v$, the
function $f_v$ has the same properties as  $f$. Therefore, by
Proposition \ref{stw8.3}, there exists a unique solution to the
problem
\[
-Aw=f_v(\cdot,w)\quad\mbox{in}\quad D,\qquad w_{|\partial_MD}=-\psi.
\]
Write $g=f_v(\cdot,w)$. Then, by the very definition of a solution, $g\in L^1_\delta(D;m)$ and
$w=-\gamma^{-1}_A(\psi)+R^Dg$.
\end{dow}

For the proof of (\ref{eq8.9}) it is  convenient to define beforehand some  subsets of the set of all  measures $\mu\in\MM_{\delta}$ whose potential $R^D\mu$  admits decomposition  of the form
\begin{equation}
\label{eq8.5}
R^D\mu=R^D f_0-h+v,
\end{equation}
where  $f_0\in L^1_\delta(D;m)$, $h$ is a harmonic function and $v$ is a function from the space  $v\in L^1(D;m)$  such that $ f(\cdot,v)\in L^1_\delta(D;m)$.

By $\RR^p(A,f)$, $ p>1$ (resp. $\RR^1(A,f)$) we denote the set of $\mu\in\MM_\delta$ such that $R^D\mu$ admits
decomposition (\ref{eq8.5}) with $h\in H_{\mathfrak S^p}$ (resp. $h\in H_{\mathfrak D}$). By $\RR_0$ we denote the set
of those $\mu\in \RR^1$ for which $h=0$ in decomposition (\ref{eq8.5}).
\begin{wn}
\label{wn8.2}
$\RR^p(A,f)=\RR_0(A,f)$ for every $p\ge1$.
\end{wn}
\begin{dow}
Of course, $\RR_0\subset \RR^p$ for all $ p\ge 1$. Let $\mu\in \RR^1$. Then $R^D\mu$ admits  decomposition (\ref{eq8.5}).
It is well known that $R^D|\mu|, R^D|f^0|$ are quasi-continuous as excessive function of the symmetric regular Dirichlet form
(see, e.g., \cite[Theorem 4.6.1]{FOT}). Moreover, by Proposition \ref{stw7.2}, $R^D|\mu|, R^D|f^0|\in\TT_0$, so they are quasi-bounded. The function $h$ is also
quasi-bounded as it is continuous on $D$ and belongs to $\TT_1$. Consequently,  $v$ is quasi-bounded.
By Corollary \ref{wn8.1}, there exists $g\in L^1_\delta(D;m)$ such that $f(\cdot,v-h+R^Dg)\in L^1_\delta(D;m)$. This implies that $\mu\in\RR_0$.
\end{dow}

\begin{uw}
\label{uw8.red}
Assume additionally that there is $r\in L^1_\delta(D;m)$ such that  $f^+(x,y)\le r(x)$ for all $x\in D, y\in\BR$.
It is immediate  that the results of \cite[Section 5]{Kl:CVPDE} are true under considered here assumptions on $f$.
Let $g\in L^1_\delta(D;m)$ be such that $f(\cdot,\eta)\in L^1_\delta(D;m)$ with $\eta=\gamma^{-1}_A(\psi)+R^Dg$.
Observe that $\mu\in\GG_\psi(A,f)$ if and only if $\mu-g\cdot m\in \GG_0(A,f_\eta)$.
By \cite[Corollary 5.12]{Kl:CVPDE}, the last one is equivalent to $\mu\in \GG_0(A,f_\eta)$. Therefore, by \cite[Theorem 5.11]{Kl:CVPDE}, $\mu\in\GG_\psi$ if and only if $\mu_c\in\GG_\psi$.
\end{uw}

\begin{tw}
\label{tw8.1}
Assume that there is $r\in L^1_\delta(D;m)$ such that
$f^+(x,y)\le r(x)$, $x\in D,\, y\in\BR$.  Then \mbox{\rm(\ref{eq8.9})} holds true for every $\psi\in L^1(\partial_MD;h^A_m)$
\end{tw}
\begin{dow}
By \cite[Corollary 5.12]{Kl:CVPDE}, $\RR_0=\GG_0$, whereas by Corollary \ref{wn8.2}, $\GG_0=\RR^1$. Therefore, it suffices to show that
$\RR^1=\GG_\psi$. It is clear that $\GG_\psi\subset \RR^1$. Let $\mu\in\RR^1$. Then $\mu$ admits decomposition (\ref{eq8.5}). By Corollary \ref{wn8.1}, there exists $g\in L^1_\delta(D;m)$ such that
$f(\cdot,w)\in L^1_\delta(D;m)$, where $w=v-h+\gamma^{-1}_A(\psi)+R^Dg$. Therefore $w$ is a solution to the problem
\[
 -Aw=f(\cdot,w)+\mu'\quad\mbox{in}\quad D,\qquad w_{|\partial_MD}=\psi
\]
with $\mu'=\mu-f_0\cdot m+g\cdot m-f(\cdot,w)\cdot m$.
Hence $\mu'\in \GG_\psi$. Consequently,  $\mu\in \GG_\psi$ by Remark \ref{uw8.red}.
\end{dow}

\begin{stw}
\label{stw8.5}
Let  $u$ be a solution to \mbox{\rm(\ref{eq8.0})}. Then for every $k>0$,
\[
\|u\|_{L^1(D;m)}+\|T_k(u)\|^2_{\breve H^1_\delta}\le 3k\lambda^{-1}(\|\psi\|_{L^1(\partial_MD;h^A_m)}
+\|f(\cdot,0)\|_{L^1_\delta(D;m)}+\|\mu\|_{TV,\delta}).
\]
\end{stw}
\begin{dow}
By Theorem \ref{tw7.1}(iii) and \cite[Theorem 3.7]{Kl:CVPDE}, for q.e. $x\in D$ we have
\begin{align*}
u(X_t)&=\psi(X_{\tau_D-})+\int_t^{\tau_D}f(X_r,u(X_r))\,dr
+\int_t^{\tau_D}\,dA^{\mu_d}_r\\
&\quad-\int_t^{\tau_D}\bar\sigma\nabla u(X_r)\,dB_r,\quad t\le \tau_D \quad P_x\mbox{-a.s.}
\end{align*}
By the Tanaka formula,
\begin{align}
\label{eq8.7}
\nonumber |u(X_t)|&=|\psi(X_{\tau_D-})|+\int_t^{\tau_D}\mbox{sgn}(u)(X_r)f(X_r,u(X_r))\,dr
+\int_t^{\tau_D}\mbox{sgn}(u)(X_r)\,dA^{\mu_d}_r\\
&\quad-\int_t^{\tau_D}\,dL^0_r-\int_t^{\tau_D}\mbox{sgn}(u)(X_r)\bar\sigma\nabla u(X_r)\,dB_r,\quad t\le\tau_D,
\end{align}
where $L^0$ is the symmetric  local time of $u(X)$ at 0, and
\begin{align*}
(|u|\wedge k)(X_t)&=(|\psi|\wedge k)(X_{\tau_D-})+\int_t^{\tau_D}\mathbf{1}_{\{|u|\le k\}}\mbox{sgn}(u)(X_r)f(X_r,u(X_r))\,dr\\
&\quad+\int_t^{\tau_D}\mathbf{1}_{\{|u|\le k\}} \mbox{sgn}(u)(X_r)\,dA^{\mu_d}_r-\int_t^{\tau_D}\,dL^0_r
+\int_t^{\tau_D}\,d\bar L^k_r\\
&\quad-\int_t^{\tau_D}\mathbf{1}_{\{|u|\le k\}}\mbox{sgn}(u)(X_r)\bar\sigma\nabla u(X_r)\,dB_r,\quad t\le\tau_D,
\end{align*}
where $\bar L^k$ is the local time of $|u(X)|$ at $k$. Integrating by parts and using the fact that $\bar L^k$ increases only when $|u|=k$ and $L^0$ increases only when $u=0$, we obtain
\begin{align*}
&(|u|\wedge k)(|u|-k)(x)+E_x\int_0^{\tau_D}|\bar\sigma\nabla T_k(u)|^2(X_r)\,dr\\
&\quad\le E_x(|\psi|\wedge k) (|\psi|-k)(X_{\tau_D-}) \\
&\qquad+E_x\int_t^{\tau_D}\{|u|\wedge k+\mathbf{1}_{\{|u|\le k\}} (|u|-k)\}\mbox{sgn}(u)(X_r)f(X_r,u(X_r))\,dr\quad\\
&\qquad+E_x\int_t^{\tau_D}\{|u|\wedge k+\mathbf{1}_{\{|u|\le k\}} (|u|-k)\}\mbox{sgn}(u)(X_r)\,dA^{\mu_d}_r+kE_x\int_0^{\tau_D}\,dL^0_r.
\end{align*}
From the above equation we conclude that
\begin{align}
\label{eq8.6}
&E_m\int_0^{\tau_D}|\bar\sigma\nabla T_k(u)|^2(X_r)\,dr \le k\Big(E_m|\psi|(X_{\tau_D-})+\|u\|_{L^1}+E_m\int_0^{\tau_D}\,dL^0_r\Big)\nonumber\\ &\qquad+2k\Big(E_m\int_0^{\tau_D}|f(X_r,u(X_r))|\,dr
+E_m\int_0^{\tau_D}\,dA^{|\mu_d|}_r\Big).
\end{align}
By (\ref{eq8.7}), monotonicity of $f$ and \cite[Theorem 3.7]{Kl:CVPDE},
\begin{align*}
\|u\|_{L^1}+E_m\int_0^{\tau_D}\,dL^0_r&+E_m\int_0^{\tau_D}|f(X_r,u(X_r))|\,dr\\&
\le E_m|\psi|(X_{\tau_D-})+E_m\int_0^{\tau_D}|f(X_r,0)|\,dr+\|\mu\|_{TV,\delta},
\end{align*}
which when combined with (\ref{eq8.6}) proves the proposition.
\end{dow}
\medskip

In the rest of the section we  assume that $D$ is of class $C^{1,1}$. It is well known (see \cite{Widman,Zhao}) that under this assumption there exist $c_1,c_2>0$ such that
\[
c_1\mbox{dist}(x,\partial D)\le R^D1(x)\le c_2 \mbox{dist}(x,\partial D),\quad x\in D.
\]

\begin{stw}
\label{stw8.6}
Let $u\in H_{\mathfrak D}$. Then $u\in L^p_\delta(D;m)$ for $p<\frac{d}{d-1}$.
\end{stw}
\begin{dow}
By Theorem \ref{tw7.1}, there is $\psi\in L^1(\partial_MD;h^A_m)$  such that $\gamma^{-1}_A(\psi)=u$. Denote by  $w_r$  the solution of the problem
\[
-Aw_r=-w_r|w_r|^{r-1}\quad\mbox{in}\quad D,\qquad (w_r)_{|\partial_MD}=\psi
\]
with some $r>1$. It exists by Proposition \ref{stw8.3}. Write
$g_r=-w_r|w_r|^{r-1}$. By the definition of a solution, $g_r\in L^1_\delta(D;m)$ and
\begin{equation}
\label{eq8.8}
w_r=\gamma^{-1}_A(\psi)+R^Dg_r.
\end{equation}
By \cite{Rakotoson}, $R^Dg_r\in L^p_\delta(D;m)$ for $ p<\frac{d}{d-1}$. Furthermore, since  $g_r\in L^1_\delta(D;m)$,  $w_r\in L^r_\delta(D;m)$ as well. Therefore, from  (\ref{eq8.8}) with $r=p$, it follows that
$\gamma^{-1}_A(\psi)\in L^p_\delta(D;m)$ for $ p<\frac{d}{d-1}$. This proves the proposition since $u=\gamma^{-1}_A(\psi)$.
\end{dow}

\begin{wn}
Let $u$ be a solution to \mbox{\rm(\ref{eq8.0})}. Then $u\in W^{1,q}_\delta(D)$ for $q<\frac{2d}{2d-1}$.
\end{wn}
\begin{dow}
By \cite{Rakotoson} and Proposition \ref{stw8.6}, $u\in L^p_\delta(D;m)$ for $p<\frac{d}{d-1}$. Combining this with Proposition \ref{stw8.5}  and using standard argument (see \cite[Lemma 4.2]{BBGGPV}) gives the desired result.
\end{dow}

\end{document}